\documentclass[authoryear]{elsarticle}

\usepackage{lineno,hyperref}
\usepackage{indentfirst}
\usepackage{url}
\usepackage{geometry}
\usepackage{amsmath}
\usepackage{amsfonts}
\usepackage{amsthm}
\usepackage{amssymb}
\usepackage{bm}
\usepackage{longtable}
\usepackage{graphicx}
\usepackage{subfigure}
\usepackage{color}
\usepackage{rotating}
\usepackage{cases}
\usepackage{indentfirst}
\usepackage{multirow}
\usepackage{natbib}

\usepackage[toc,page]{appendix}

\modulolinenumbers[5]

\newtheorem{remark}{\bf Remark}

\newtheorem{lemma}{\bf Lemma}
\newtheorem{example}{\bf Example}

\newtheorem{corollary}{\bf Corollary}
\newtheorem{proposition}{\bf Proposition}
\newtheorem{conjecture}{\bf Conjecture}
\newtheorem{theorem}{\bf Theorem}
\newtheorem{definition}{\bf Definition}

\newtheorem{thm}[theorem]{\bf Theorem}
\newtheorem{lem}[lemma]{\bf Lemma}

\newtheorem{defn}[definition]{\bf Definition}

\newcommand{\bN}{\ensuremath{\mathbb{N}}}

\newcommand{\bR}{\ensuremath{\mathbb{R}}}

\newcommand{\eps}{\ensuremath{\epsilon}}

\newcommand{\beq}{\begin{equation}}
\newcommand{\eeq}{\end{equation}}

\newcommand{\ip}[2]{\ensuremath{\langle {#1},{#2} \rangle}}
\newcommand{\abs}[1]{\ensuremath{\left\vert{#1} \right\vert}}
\DeclareMathOperator{\lspan}{span}
\DeclareMathOperator{\rk}{rank}

\newcommand{\kk}{K_{\alpha}}
\newcommand{\kx}{K_{\alpha}(X)}
\newcommand{\km}{K_{1/5}}
\newcommand{\ex}{\overline{x}}

\begin{document}

\begin{frontmatter}

\title{New Upper Bounds for Equiangular Lines by Pillar Decomposition
}


\author{Emily J. King}
\ead{king@math.uni-bremen.de}
\address{Faculty 3: Mathematics / Computer Sciences\\\vspace*{-1mm}University of Bremen,  28359 Bremen, Germany}

\author{Xiaoxian Tang\corref{mycorrespondingauthor}}
\cortext[mycorrespondingauthor]{Corresponding author}
\ead{xiaoxian@math.tamu.edu}
\ead[url]{sites.google.com/site/rootclassification}
\address{Department of Mathematics\\Texas A\&M University,  College Station, TX 77843-3368, USA}




\begin{abstract}
We derive a procedure for computing an upper bound on the number of equiangular lines  in various Euclidean vector spaces by generalizing  
the classical pillar decomposition developed by \citep{LS1973}; namely, we use linear algebra and combinatorial arguments to bound the number of vectors within an equiangular set which have inner products of certain signs with a negative clique.  After projection and rescaling, such sets are also certain spherical two-distance sets, and semidefinite programming techniques may be used to bound the size. Applying our method, we prove new relative bounds for the angle $\arccos(1/5)$. Experiments show that 
our relative bounds for all possible angles are considerably less than the known semidefinite programming bounds for a range of larger dimensions. Our computational results also show 
an explicit bound on the size of a set of equiangular lines in $\bR^r$ regardless of angle, which is strictly less than the well-known Gerzon's bound if $r+2$ is not a square of an odd number: 
\begin{equation*}
\begin{cases}
\frac{4r\left(m+1\right)\left(m+2\right)}{\left(2m+3\right)^2-r}& r=44, 45, 46, 76, 77, 78, 117, 118, 166, 222, 286, 358\\
\\
\frac{\left(\left(2m+1\right)^2-2\right)\left(\left(2m+1\right)^2-1\right)}{2}& \text{other}~r~\text{between}~44~\text{and}~400, 
\end{cases}
\end{equation*}
where $m$ is the largest positive integer such that $(2m+1)^2\leq r + 2$.

\end{abstract}


\end{frontmatter}

\section{Introduction}\label{introduction}

In this paper, we are concerned with the maximum number of equiangular lines in Euclidean vector spaces: 

\bigskip

\noindent
{\bf Problem Statement.} For a given integer $r$ with $r\geq 2$, what is the maximum number of distinct lines in $r$-dimensional Euclidean space ${\mathbb R}^r$ such that the angle between each pair of lines equals $\arccos(\alpha)$ for some $0<\alpha<1$?

\bigskip
\noindent
By selecting a unit vector in each line in a set of equiangular lines, we can formally define it as an equiangular set of unit vectors. 

\begin{definition}\label{def:def1}
We say $X=\{x_1, \ldots, x_s\}\subset {\mathbb R}^r$ is a {\em set of equiangular lines} (or, simply \emph{equiangular}) if  for some $0 < \alpha < 1$,
\begin{equation}\label{def1}
\langle x_i, x_j \rangle ~=~
\begin{cases}
1, & i = j\\
\pm\alpha, & i \neq j.
\end{cases}
\end{equation}
\end{definition}

\noindent 
By slight abuse of terminology, we will say that vectors which satisfy the equality~(\ref{def1}) are equiangular with \emph{angle $\alpha$}, even though the actual angle is $\arccos(\alpha)$. The ambiguity of the sign of the inner product is due to the choice of a unit vector in each line.


Fix the dimension $r$ $(r\geq 2)$ and the angle $0<\alpha <1$.  We define $s_\alpha(r)$ to be the maximum cardinality of an equiangular set in $\bR^r$ with angle $\alpha$.  Further $s(r)$ is defined to be the maximum cardinality of any equiangular set in $\bR^r$, that is,
\[
s(r)~\triangleq~ \max_{\alpha > 0} ~s_\alpha(r).
\]

\noindent 
Now the problem statement can be precisely rewritten: 

\bigskip
\noindent
{\bf Problem Statement$\mathbf{'}$.} For a given integer $r$ with $r\geq 2$,  what is $s(r)$?

\bigskip
\noindent
{\bf Motivation.}
The problem of finding the maximum number of equiangular lines has been studied for at least 70 years \citep{Haan48}. Equivalence classes of sets of equiangular lines are equivalent to so-called \emph{two-graphs} (not be be confused with $2$-graphs) and are intricately connected with many problems in algebraic graph theory \citep{GodRoy01}. They are also equivalent to \emph{spherical codes} with particular angle sets \citep{delsarte1977}. In a seminal paper \citep{GrassPack}, the authors credit a post to a newsgroup in 1992 from an oncologist named Julian Rosenman for their interest in the field. The post asked the best way to separate laser beams going through a particular tumor, which can be thought of as asking for in-some-sense optimal packings of lines in $\bR^3$.  From the point of view of applications, in certain extremal cases, equiangular sets have further desirable properties with respect to data analysis and coding theory \citep{GrassPack,StH03}.  These special sets are called {\em equiangular tight frames} (ETF). An ETF is an equiangular set $\{x_1, \ldots, x_s\}\subset {\mathbb R}^r$ such that 
for any $y \in \bR^r$, the following Parseval-like equality holds
\begin{equation*} 
\label{eqn:frame}
\frac{s}{r} ~\langle y, y \rangle ~=~ \sum_{j=1}^s \langle {y}, {x_j} \rangle ^2.
\end{equation*}
ETFs solve a packing problem in Grassmannian space \citep{BKo06,GrassPack,DHST08}, are known to be optimally robust to erasures \citep{StH03,Bod07}, and further have optimal coherence which is related to the appropriateness of using a set of vectors for sparse coding \citep{DE03,BDE09}.

The theory of equiangular lines and frames is related to linear algebra (e.g., existence of certain matrices \citep{LS1973,VaSei66,STDH07}), combinatorial group theory (e.g., difference sets \citep{XZG05,DiFe07}), geometry (e.g., regular spherical polytopes \citep{Cox}), graph theory (e.g., [regular] two-graphs and strongly regular graphs \citep{HoPa04,VaSei66}), combinatorial designs (e.g., Steiner systems \citep{FMT12}), Jacobi polynomial expansions (\citep{delsarte1977}), and more.

Table~\ref{known} presents the currently known $s(r)$ for dimensions $2 \leq r \leq 43$.  
\begin{table}[h]
\caption{Maximum number of equiangular lines for small dimensions \citep{LS1973, Wal09, BaYu14, GKMS15,Yu15,AzMa16,Sz2017,GG2018,GrY18}}\label{known}
\begin{center}
\begin{tabular}{c||ccccccc}
$r$&             $2$ & $3$ & $5$  & $6$ & $7$--$13$ & $14$   \\
$s(r)$& $3$ & $6$ & $10$ & $16$&$28$& $28$--$29$  \\
\hline
\hline
$r$& $15$ &$16$ &               $17$ &         $18$&     $19$&              $20$ \\
$s(r)$& $36$ & $40$--$41$ & $48$--$49$ &$54$--$60$&$72$--$75$  & $90$--$95$  \\
\hline
\hline
$r$&     $21$ & $22$ &$23$--$41$ & $42$ & $43$ \\
$s(r)$ &$126$ & $176$& $276$& $276$--$288$  & $344$ 
\end{tabular}
\end{center}
\end{table}
\noindent
An attractive direction of research is to develop a general method to compute $s(r)$ or a bound on $s(r)$ for any $r\geq 44$. So far, for any $r$ $(r\geq 44)$, we only know \cite[Corollary 2.8]{GKMS15}
\beq\label{eqn:greavebnd}
\frac{32r^2+328r+296}{1089}\leq s(r) \leq \frac{r(r+1)}{2}.
\eeq
Here the upper bound $r(r+1)/2$ is known as the famous Gerzon's bound (from private discussions with Gerzon mentioned in \citep{LS1973}). 
It is also well-known that Gerzon's upper bound can be sharpened considerably for certain $r$. For instance, $s(43)=344<(43\times 44)/2=946$. 

One possible way to get a 
non-trivial upper bound is to consider $s_\alpha(r)$ for a fixed angle $\alpha$. 
A nice classical result is that $s(r)$ $(r>3)$  can be solved by determining finitely many 
$s_\alpha(r)$ where $1/\alpha$ is an odd  integer bounded by $\sqrt{2r}$ (Proposition~\ref{angles}). Thus, throughout the paper, we assume $1/\alpha$ is an odd integer which is greater than $3$. 
An upper bound for $s_\alpha(r)$ is often called a {\em relative bound}. (We note that in some papers, {\em relative bound} refers specifically to the analog of Gerzon's bound for a particular angle.)
Another theorem of note is Theorem 4.5 in \citep{LS1973}, 
which determines $s_{1/3}(r)$ completely by decomposing an equiangular set into {\em pillars} (see the precise definition in Section~\ref{reviewseidel}) and studying the algebraic structure of each pillar and also the combinatorial structure when all pillars are non-empty.  
However, by the same spirit, the next interesting case $\alpha=1/5$
is only partially solved and there is a long-standing conjecture \citep[Conjecture 5.8]{LS1973}.
The best known $s_{1/5}(r)$ is summarized in \citep[Table 4]{GKMS15}. 
For $\alpha\leq 1/7$, following the classical method, one might need to characterize the connected simple graphs with maximum eigenvalue 
$3$ (or $>3$). The last sentence in \citep{neumaier1989} says this requires substantially stronger techniques.  The good news is that relative bounds for general $\alpha$ can be computed by semidefinite programming (SDP) \citep{BaYu14}. The best known non-trivial relative bounds and upper bound of $s(r)$ for $44\leq r\leq 136$ that existed before this paper was originally released can be found in \citep[Table 3]{BaYu14}. Notice that for $r>136$, this SDP method might give a bound which is greater than Gerzon's bound. 

\bigskip
\noindent
{\bf Contributions.}
In this paper, our main contribution is a universal  procedure for computing a non-trivial upper bound for general  dimension $r$. 
Our main contributions have three stages, see (C1), (C2), and (C3). 
\begin{itemize}
\item[(C1)] We derive a procedure (Subsection \ref{pipeline}) for computing an upper bound of an equiangular  set  $X\subset {\mathbb R}^r$ with angle $\alpha$. Our method is to  decompose  $X$  into finitely many equivalence classes with respect to 
a fixed $K$-base (Definition \ref{def:kbase}, a maximal negative clique). Here each equivalence class generalizes the concept of a pillar defined in \citep[Page 501, Section 4]{LS1973}. We 
prove the number of equivalence classes in $X$ and an upper bound for each equivalence class (Theorems~\ref{general},~\ref{thm:ec}, and~\ref{th:s2}). 

\item[(C2)] After a more careful analysis of the pillar decompositions (Theorems~\ref{general},~\ref{thm:ec}, and~\ref{th:s2}), we provide new relative bounds for the angle $\alpha =1/5$ (Theorem \ref{thm:rb5} and Corollary \ref{cry:rb5}). 

\item[(C3)] By applying Theorems~\ref{general},~\ref{thm:ec},~\ref{th:s2}, and~\ref{thm:rb5}, we compute upper bounds of equiangular sets for 
$44\leq r\leq 400$. The computational results show an explicit upper bound:
\begin{equation*}
\begin{cases}
\frac{4r\left(m+1\right)\left(m+2\right)}{\left(2m+3\right)^2-r}, & r=44, 45, 46, 76, 77, 78, 117, 118, 166, 222, 286, 358\\
\\
\frac{\left(\left(2m+1\right)^2-2\right)\left(\left(2m+1\right)^2-1\right)}{2},&  \text{other}~r~\text{between}~44~\text{and}~400. 
\end{cases}
\end{equation*}
This bound is strictly less than Gerzon's bound if $r+2$ is not a 
square of an odd number (Theorem~\ref{thm:main}).  The result leads us to a conjecture on a new general upper bound (Conjecture~\ref{conjecture}).  
\end{itemize}

We note that our approach using the pillar decomposition can be seen as being similar to the methods of \citep{BDKS18}, which were independently developed. In \citep{BDKS18}, given an equiangular set, a weighted graph with vertices corresponding to vectors in the set and edge weights being the values of the corresponding inner products is constructed.  Using Ramsey theory, for large enough graphs, there must exist a clique of positive edges.  They proceed using orthogonal projections onto certain sets, like the orthogonal complement of the large clique, to transform their problem about equiangular lines to a related problem about certain two-angle spherical codes. (See also~\citep{Goss04}.) Our methods are related to finding maximal negative cliques and characterizing inner products connecting to the clique via orthogonal projections.  The size of positive cliques is bounded by the dimension and negative cliques by a function of the angle, thus yielding a very different structure. The main theorem of their paper concerning equiangular lines is a collection of asymptotic bounds, namely:
\begin{thm}\cite[Theorem 1.1]{BDKS18}\label{thm:asymp}
Fix $\alpha \in (0,1)$.  For $r$ sufficiently large relative to $\alpha$, the maximum number of lines in $\bR^r$ with angle $\alpha$ is exactly $2r-2$ if $\alpha = 1/3$ and at most $1.93r$ otherwise.
\end{thm}
At first glance, this seems to contradict \eqref{eqn:greavebnd}, $(32r^2+328r+296)/1089\leq s(r) \leq r(r+1)/2$ for $r \geq 44$; however, it is important to note that this bound in Theorem~\ref{thm:asymp} is asymptotic in dimension \emph{relative to a fixed $\alpha$}.  That is, for large enough $r$ (where ``large enough'' depends on $\alpha$) $s_\alpha(r) \leq 2r -2$; however, it is not true in general that $s(r) \leq 2r-2$ holds for large enough $r$.  As an example, we can see from our computational results \citep[Table.pdf]{link2018} or in Figure~\ref{fig:bound5} that the upper bound on the size of an equiangular set of vectors in $\bR^r$ for $44 \leq r \leq 400$ with angle $1/5$ grows relatively slowly with $r$, while at the same time being significantly smaller than the known bounds for smaller angles.

\bigskip
\noindent 
{\bf Organization.}
The rest of the paper is organized as follows. 
In Section~\ref{reviewsdp}, we review the definition of spherical two-distance sets and the fact that an upper bound of a spherical two-distance set can be solved by 
semidefinite programming (SDP) \citep{BaYu13}. 
In Section~\ref{pd}, we decompose a given equiangular set $X$  into finitely many equivalence classes with respect to 
a fixed $K$-base (Definitions \ref{def:kbase}-\ref{def:ec}).  According to the size of a $K$-base, we 
show the number of equivalence classes in $X$ and an upper bound for each equivalence class (Theorems~\ref{general},~\ref{thm:ec} and~\ref{th:s2}). We also provide a procedure (Subsection \ref{pipeline}) and illustrate an example (Example \ref{ex:rb7}) for computing an upper bound for given dimension $r$ and angle $\alpha$. 
By further analyzing the pillar decomposition in Section~\ref{rb5}, we are able to prove a new relative bound for the angle $\alpha=1/5$ (Theorem \ref{thm:rb5} and Corollary \ref{cry:rb5}). 
In Section~\ref{main}, by applying the procedure, we compute upper bounds for the dimensions $44\leq r\leq 400$ and conclude the computational results in Theorem \ref{thm:main}. 
In Section~\ref{experiment}, we compare the new (computable) relative bounds for $\alpha=1/5, 1/7$ and the SDP bounds \citep[Theorem 3.1]{BaYu13} by experiments. We also interpret the experimental/computational details in this section.   The results in Sections~\ref{pd} and~\ref{rb5} which require technical proofs are proven in the order they appear in the text in Appendices~\ref{App:AA}--\ref{app:AD}.

\section{Review of spherical two-distance sets and  SDP bounds}\label{reviewsdp}
In this section, we review spherical two-distance sets and the fact that an upper bound of a spherical two-distance set can be determined by 
semidefinite programming (SDP). 
A spherical two-distance set is a more general concept than an equiangular set. The SDP method is closely related to Delsarte's method \citep{delsarte1977, musin2009}  and 
harmonic analysis in coding theory \citep{BV2008}. 
We provide 
\citep{Yu14} as a good survey for the interested readers since we will only repeat the key results without proof here. The main point we hope to highlight here is that an upper bound of equiangular sets in ${\mathbb R}^r$ is computable  for any $r>3$, see 
Theorem~\ref{sdp} and Proposition~\ref{angles} below. 
\begin{defn}
We say $X=\{x_1, \ldots, x_s\}\subset {\mathbb R}^r$ is a {\em spherical two-distance set with mutual inner products $\alpha, \beta$}  if  
\begin{equation*}
\langle x_i, x_j \rangle~ =~
\begin{cases}
1,& i = j\\
\alpha~ \text{or}~ \beta, & i \neq j.
\end{cases}
\end{equation*}
\end{defn}

For any spherical two-distance set $X$, we use the notation $|X|$ to denote the cardinality of $X$.  
We use the notation $s\left(r, \alpha, \beta\right)$ to denote the maximum cardinality of a spherical two-distance set with mutual inner products $\alpha, \beta$ in ${\mathbb R}^r$. 
\begin{thm}\label{sdp}\citep[Theorem 3.1]{BaYu13}
Suppose $X\subset {\mathbb R}^r$ is a spherical two-distance set with mutual inner products $\alpha, \beta$.  
An upper bound of $|X|$ is given by 
the solution of a semidefinite programming (SDP) problem.
\end{thm}

The concrete SDP formulation can be found in \citep[Theorem 3.1]{BaYu13}. 
As we can see, 
an equiangular set $X$  in ${\mathbb R}^r$ with angle $\alpha$ is a special spherical two-distance set with mutual inner products $\alpha, -\alpha$; that is, $s_\alpha(r) = s(r,\alpha,-\alpha)$.
So an upper bound of $|X|$ for a given $\alpha$ can be computed by running SDP tools. 
Proposition \ref{angles}, which is a direct corollary of \cite[Theorem 2]{LRS1977}, shows an upper bound of $|X|$ is either $2r+3$ or given by $s_\alpha(r)$ with $\alpha=1/(2L-1)$ for finitely many possible $L$'s.  

\begin{proposition}\label{angles}
 For any equiangular set $X\subset {\mathbb R}^r$ with angle $\alpha$, if $|X|>2r+3$, then
$\alpha=1/(2L-1)$ for some $L$ such that  $2\leq L\leq (1+\sqrt{2r})/2$.
 \end{proposition}

We remark that $|X|\leq 2r-2 ~(< 2r + 3)$ when $\alpha=1/(2\cdot 2-1)=1/3$ and $r\geq 15$ \citep[Theorem 4.5]{LS1973}. 
Thus by Theorem~\ref{sdp} and Proposition~\ref{angles}, for any $r\geq 15$, 
an upper bound of $|X|$ is either $2r+3$ or 
the maximum of the SDP bounds for $\alpha =1/5, \ldots, 1/(2L-1)$. 
Big progress was made in this spirit in proving the relative bound $276$ for $24\leq  r\leq 60$ and $\alpha=1/5$ \citep{BaYu14}.  
However, it is seen that 
the SDP bounds for $\alpha=1/5$ are greater than Gerzon's bound for $r=137$--$139$  
\citep[Table 3]{BaYu14}. We also provide evidence in Section~\ref{experiment} that 
for $\alpha=1/5, 1/7$,  the SDP bound does not guarantee a non-trivial upper bound if $r$ is sufficiently large, and we expect similar behavior for smaller values of $\alpha$.  In the rest of the paper, we focus on improving the relative bounds with help of the pillar decomposition. 
We remark that all our results  except Theorem \ref{general} hold for any $0<\alpha<1$. However, we are only interested in the $\alpha$'s such that $1/\alpha$ is an odd number greater or equal to $3$.

\section{Pillar decomposition}\label{pd}

\subsection{Gramian matrices}\label{gramian}


For a set of vectors $X=\{x_1, \hdots, x_s\}$ in ${\mathbb R}^r$ $(r\geq 2)$, 
we also denote by $X$ the $r\times s$ matrix with $x_1, \hdots, x_s$ as its column vectors.  
We begin by defining the Gramian matrix of $X$.


\begin{definition}\label{def:gramian}
For any  $X = \{ x_1, \hdots, x_s\} \subset \bR^r$, the \emph{Gramian matrix} of $X$, 
denoted by $G(x_1, \ldots, x_s)$ or 
$G(X)$, is the matrix of mutual inner products of $x_1, \ldots, x_s$, namely
\[
G(X)~=~X^{\top}X~=~
\left(\begin{array}{cccc}\ip{x_1}{x_1} & \ip{x_1}{x_2} & \hdots & \ip{x_1}{x_s} \\  \ip{x_2}{x_1} & \ip{x_2}{x_2} & \hdots & \ip{x_2}{x_s} \\ \vdots & \vdots & \ddots & \vdots \\ \ip{x_s}{x_1} & \ip{x_s}{x_2} & \hdots & \ip{x_s}{x_s} \\  \end{array}\right)_{s\times s}.
\]
\end{definition}

\noindent
It follows directly from Definition~\ref{def:gramian} that $G(X)$ is symmetric and positive semidefinite  for any finite set $X\subset {\mathbb R}^r$. The following lemma is standard and shows that $X$ is a set of linearly independent vectors if and only if $G(X)$ is non-singular.
\begin{lem}[see, for example, \citep{HoJo12}]\label{lem:gramspan} 
For any  $X = \{ x_1, \hdots, x_s\} \subset \bR^r$, $\rk G(X) = \dim \lspan X$.  Further, a linear dependence relation of the columns is one of the vectors and vice versa.
\end{lem}

\subsection{Switching equivalent equiangular sets}

Recall from the introduction that we denote equiangular lines in ${\mathbb R}^r$ by a finite set of unit equiangular vectors $X=\{x_1, \ldots, x_s\}$.  
We note that when passing from lines to vectors, we must make a choice of one of the two unit vectors which span the same line.  
In particular, two unit vectors $-x_i$ and $x_i$ denote the same line. 
However, this choice affects the signs of the inner products.  
 If two sets of vectors represent the same set of lines,  we say they are \emph{switching equivalent} \citep{VaSei66,GodRoy01}. More generally and more precisely, we have the definition below. 
 
\begin{definition}\label{def:se}
Two sets of unit vectors $X$ and $Y$ in ${\mathbb R}^r$ are {\em switching equivalent} 
if there exist a diagonal $(1,-1)$-matrix $B$ and a permutation matrix $C$ such that
\[(CB)^{\top}\cdot G(X)\cdot (CB) ~= ~G(Y).\]
We also say $G(X)$ is {\em switching equivalent} to $G(Y)$, denoted by $G(X)\cong G(Y)$.
\end{definition}

\begin{lemma}\label{lm:seg}
Let $X$ and $Y$ be two sets of unit vectors  in ${\mathbb R}^r$.
If $G(X)\cong G(Y)$, then  $G(X)$ and $G(Y)$ have the same eigenvalues. 
\end{lemma}
\begin{proof}
If $G(X)\cong G(Y)$, then by Definition~\ref{def:se}, they are orthogonally similar and hence have the same eigenvalues.
\end{proof}

\begin{lemma}\label{lm:switch}
Let $X$ and $Y$ in ${\mathbb R}^r$ be two equiangular sets with the same angle $\alpha$. 
 If $|X|=|Y|$ and we can order the vectors in 
$X$ and $Y$ as $ \{x_1, \hdots, x_s\}$ and  $\{y_1, \hdots, y_s\}$ respectively such that
$x_i=\pm y_i$ for $i=1, \ldots, s$, then $X$ and $Y$ are switching equivalent  $($i.e., $G(X)\cong G(Y))$.
\end{lemma}
\begin{proof}
Note there exist a diagonal $(1,-1)$-matrix $B$ and a permutation matrix $C$ such that 
$XCB=Y$ and hence $(CB)^{\top}\cdot G(X)\cdot (CB) ~= ~G(Y)$. By Definition \ref{def:se}, $X$ and $Y$ are switching equivalent. 
\end{proof}
We further remark that if $G(X)\cong G(Y)$, there need not exist  a diagonal $(1,-1)$-matrix $B$ and a permutation matrix $C$ such that $XCB=Y$.  For example, if $O$ is an orthogonal matrix, then $G(OX) = G(X)$.  The key idea is that given a positive semidefinite matrix $G$, we can always factor it (via, for example, an eigendecomposition) as $G = X^T X$ so that $G = G(X)$.  Thus we will work on the level of Gramian matrices.

\subsection{Base size and $K$-base}

 Before we investigate more structures associated to equiangular sets,  we  provide a basic fact (Lemma \ref{pp:basic}) from linear algebra.
 Following Seidel's spirit,  we will often decompose matrices using building blocks of the $s \times s$ identity matrix $I_s$ and the $s \times s$ all-one matrix $J_s$, which we will denote by $I$ and $J$, respectively, when $s$ is clear from context.  


\begin{lemma}\label{pp:basic}
Consider an $s\times s$ matrix $G$. If all diagonal entries of $G$ are the same, say $a$, and all 
off-diagonal entries are the same, say $b$, then $G =  (a-b) I_s + b J_s$, and hence $G$ has a simple eigenvalue $\lambda_1 = a+(s-1)b$ and an eigenvalue $\lambda_2=a-b$ with multiplicity $s-1$. 
\end{lemma}

Following the above lemma, we present an easy observation about the structure of Gramian matrices. As what has been pointed out at the very beginning of \citep[Section 4]{LS1973}, we have Proposition~\ref{lm:easy}.
Here we generalize the original setting in \citep{LS1973} since we consider the Gramian matrices of switching equivalent equiangular sets. 
 \begin{proposition}\label{lm:easy}
 If there exist $k \geq 2$ equiangular vectors $p_1, \ldots, p_k$ with angle $\alpha$ such that \[G(p_1, \ldots, p_k)~\cong~ (1+\alpha) I -\alpha J,
\] 
then $k\leq (1/\alpha)+1$. 
Furthermore, if $k<(1/\alpha)+1$, then the vectors  $p_1, \ldots, p_k$ are linearly independent, and if
$k=(1/\alpha)+1$, then the vectors  $p_1, \ldots, p_k$ are linearly dependent.  When $k=(1/\alpha)+1$ and $G(p_1, \ldots, p_k)~=~ (1+\alpha) I -\alpha J$,  the linear dependence relation is $\sum_{j=1}^k p_j =0$. 
\end{proposition}
\begin{proof}
By Lemma \ref{lm:seg},  $G(p_1, \ldots, p_k)$  and $(1+\alpha) I -\alpha J$ have same eigenvalues and thus by Lemma \ref{pp:basic}, $G(p_1, \ldots, p_k)$ has a simple eigenvalue $\lambda_1=1-(k-1)\alpha$ and 
an eigenvalue $\lambda_2=1+\alpha>0$ with multiplicity $k-1$. Since the Gramian matrix is positive semidefinite, the eigenvalue $\lambda_1$   should be non-negative. So 
\[\lambda_1~=~1-(k-1)\alpha\geq 0 \;\;\text{or, equivalently,}  \;\;k\leq (1/\alpha)+1.\] 
Furthermore, if $k<(1/\alpha)+1$, then the eigenvalue $\lambda_1$ satisfies \[\lambda_1=1-(k-1)\alpha~=~\alpha \left((1/\alpha)+1-k\right)~>~0.\]
Both eigenvalues are nonzero, so the Gramian matrix is full rank and hence by Lemma \ref{lem:gramspan}, $p_1, \ldots, p_k$  are linearly independent. 
If $k=(1/\alpha)+1$, we similarly have $\lambda_1=0$. So  $p_1, \ldots, p_k$ are linearly dependent. If further $G(p_1, \ldots, p_k)~=~ (1+\alpha) I -\alpha J$, then the all-ones vectors is an eigenvector for the eigenvalue $0$, yielding the linear dependence relation (Lemma~\ref{lem:gramspan}).
\end{proof}

\noindent
Proposition \ref{lm:easy} inspires two new concepts: base size and $K$-base (Definition \ref{def:kbase}). 
\begin{definition}\label{def:kbase}
Let $X\subset {\mathbb R}^r$ be an equiangular set  with angle $\alpha$. The {\em base size} of $X$, $\kx$ is defined as  
\begin{equation}\label{eq:k}
\kx\;\triangleq\; \max\{k \in {\mathbb N}^{+}|\exists~ p_1, \ldots, p_k\in X~\text{s.t.}~G(p_1, \ldots, p_k)\cong (1+\alpha) I -\alpha J\}.
\end{equation}
Let $K=\kx$. If $K$ vectors $p_1, \ldots, p_K\in X$ satisfy $G(p_1, \ldots, p_K)\cong (1+\alpha) I -\alpha J$,
 then we say $\{p_1, \ldots, p_K\}$ is a {\em $K$-base} of $X$. 
\end{definition}
It follows from Definitions~\ref{def:se} and~\ref{def:kbase} that if $\{p_1, \ldots, p_K\}$ is a $K$-base of $X$,  then $\{p_1, \ldots, p_K\}$ is switching equivalent to 
an equiangular set $\{\tilde{p}_1, \ldots, \tilde{p}_K\}$, which has $G(\tilde{p}_1, \ldots, \tilde{p}_K)= (1+\alpha) I -\alpha J$, 
and further $X$ is then also switching equivalent to $(X \backslash\{p_1, \ldots, p_K\}) \bigcup \{\tilde{p}_1, \ldots, \tilde{p}_K\}$.  Thus we assume without loss of generality that
\begin{equation}\label{eq:skbase}
G(p_1, \ldots, p_K)= (1+\alpha) I -\alpha J.
\end{equation}
That means for any $1\leq i<j\leq K$, $\langle p_i, p_j\rangle=-\alpha$.
Using graph theory terminology, a $K$-base is a \emph{negative clique}.  Lemma \ref{lm:easy} shows that 
the size of a $K$-base, that is $\kx$, is at most $(1/\alpha)+1$.
We next show that 
$\kx$ is at least $2$ if $|X|\geq 2$ (Proposition \ref{lm:easy2}).

\begin{proposition}\label{lm:easy2}
Let $X\subset {\mathbb R}^r$ be an equiangular set with angle $\alpha$. 
If $|X|\geq 2$, then $\kx\geq 2$.
\end{proposition}
\begin{proof}
Take two different vectors $x_1, x_2\in X$.  
If $\langle x_1, x_2\rangle=-\alpha$, then by Definition \ref{def:kbase}, 
$\kx\geq 2$. If $\langle x_1, x_2\rangle=\alpha$, then $\langle -x_1, x_2\rangle=-\alpha$.
By Lemma \ref{lm:switch}, 
 $\{x_1, x_2\}$ is switching equivalent to $\{-x_1, x_2\}$. So 
 $G(x_1, x_2)\cong G(-x_1, x_2)=(1+\alpha)I-\alpha J$. Again,  by Definition \ref{def:kbase}, 
$\kx\geq 2$. 
\end{proof}

\subsection{Equivalence classes w.r.t.~a $K$-base}

 Let $X$ in ${\mathbb R}^r$ be an equiangular set  with angle $\alpha$. Let $K=\kx$. We fix a $K$-base  $\{p_1, \ldots, p_K\}$ of $X$ satisfying the condition (\ref{eq:skbase}). 
Let $P$ be the linear subspace of ${\mathbb R}^r$ spanned by $p_1, \ldots, p_K$ and let ${P}^{\bot}$ be the orthogonal complement of $P$ in ${\mathbb R}^r$. 
For any  $x\in X\backslash \{p_1, \ldots, p_K\}$, 
we first present Proposition \ref{lm:project} to describe the projection of 
$x$ onto the subspace $P$. Note that part (1) originally comes from \citep{LS1973}.  
 \begin{proposition}\label{lm:project}
For  any $x\in X\backslash \{p_1, \ldots, p_K\}$, suppose $(\langle x, p_1\rangle, ~\ldots, ~\langle x, p_K\rangle)^{\top} = \alpha\cdot \epsilon$,
where $\epsilon$ is a $(1, -1)$-vector in ${\mathbb R}^K$, namely
\[\epsilon~=~\left(\epsilon^{(1)}, ~\ldots,~ \epsilon^{(K)}\right)^{\top},\;\;\;\epsilon^{(i)}~=~\pm 1.\]
 If we decompose $x$ in $P$ and $P^{\bot}$  as $x = h + c$, where $h \in P$ and $c \in P^{\bot}$,
then 
\begin{itemize}
\item[{\em (1)}] for $K=(1/\alpha)+1$,  
$h~=~\frac{1}{K}\left(\epsilon^{(1)}p_1+\cdots+\epsilon^{(K)}p_K\right)$
and $\langle h, h\rangle~=~\alpha$, and 
\item[{\em (2)}]for $K<(1/\alpha)+1$, 
$h=a^{(1)}p_1+\cdots+a^{(K)}p_K$, 
where
\begin{equation}\label{eq:h}
(a^{(1)}, \ldots, a^{(K)})^{\top} ~=~ \alpha~\left((1+\alpha)I-\alpha J\right)^{-1}~\epsilon,
\end{equation}
and 
$\langle h, h\rangle$ is given by 
\begin{equation}\label{eq:l}
 \ell(K,n)~=~ \alpha^2 \frac{4\alpha\cdot n\cdot\left(n-K\right) +\left(1+\alpha\right)K}{(1+\alpha)\left(1+\alpha -K\alpha\right)},
\end{equation}
where $n$ is the number of positive signs among   $\epsilon^{(1)}, \ldots, \epsilon^{(K)}$. 
Furthermore, we have
\begin{equation}\label{eq:beta}
\begin{cases}
0 < \ell(K,n) < \alpha,  & K-\frac{(1/\alpha)+1}{2} <n\leq \lfloor K/2\rfloor\\
\ell(K,n) ~=~\alpha,& n= K-\frac{(1/\alpha)+1}{2}\\
\alpha <\ell(K,n) < 1,  & 1\leq n <K- \frac{(1/\alpha)+1}{2}.
\end{cases}
\end{equation}
\end{itemize}
\end{proposition}
\begin{proof}
Part (1) was shown in \citep[pages 501--502]{LS1973}.
We only prove part (2). 
By Proposition \ref{lm:easy}, if $K<(1/\alpha)+1$, then $p_1, \ldots, p_{K}$ are linearly independent. 
So $p_1, \ldots, p_{K}$ form a basis of $P$. 
Since $h\in P$, we can assume 
$h= \sum_{j=1}^{K}a^{(j)}p_j$.
Note by (\ref{eq:skbase}), $\langle p_i, p_j\rangle = -\alpha$ $(i\neq j)$. 
So for each $i=1, \ldots, K$, we have
\begin{equation}\label{eq:pp41}
\langle h, p_i\rangle ~=~\langle \sum_{j=1}^{K}a^{(j)}p_j, p_i \rangle ~=~a^{(i)}\cdot\langle p_i, p_i\rangle + \sum_{j\neq i}a^{(j)}\cdot\langle p_j, p_i\rangle~=~a^{(i)} -\alpha\cdot \sum_{j\neq i}a^{(j)}.
\end{equation}
On the other hand, since $x=h+c$, $h \in P$ and $c \in P^{\bot}$, we have
\begin{equation}\label{eq:pp42}
\langle h, p_i\rangle  ~=~\langle x, p_i\rangle ~=~\alpha\cdot \epsilon^{(i)}.
\end{equation}
Note the left-hand sides of equalities (\ref{eq:pp41}) and (\ref{eq:pp42}) are the same. So we have a system of $K$ linear equations in $a^{(1)},\cdots, a^{(K)}$ which has a unique solution
\[a ~=~ \left(a^{(1)}, \ldots, a^{(K)}\right)^{\top} ~=~\alpha~\left((1+\alpha)I-\alpha J\right)^{-1}~\epsilon.\]
We denote  the $r \times K$ matrix with column vectors $p_1, \ldots, p_K$ by $\left( \begin{array}{cccc}p_1 & p_2 & \hdots & p_K \end{array}\right)$. 
Then it is straightforward to calculate by linear algebra  that 
\begin{align*}
\langle h, h\rangle &~=~  \langle \sum_{j=1}^{K}a^{(j)} p_j, \;\;\sum_{j=1}^{K}a^{(j)} p_j\rangle \nonumber \\
&~=~a^{\top}\left( \begin{array}{cccc}p_1 & p_2 & \hdots & p_K \end{array}\right)^{\top}\left( \begin{array}{cccc}p_1 & p_2 & \hdots & p_K \end{array}\right)a  \nonumber\\
&~=~a^{\top}\left(\left(1+\alpha\right)I-\alpha J\right)a  \nonumber\\
&~=~\alpha\epsilon^{\top} \left(\left((1+\alpha)I-\alpha J\right)^{-1}\right)^{\top}~\left((1+\alpha)I-\alpha J\right)~\alpha\left((1+\alpha)I-\alpha J\right)^{-1}\epsilon  \nonumber\\
&~=~\alpha^2~\epsilon^{\top}  \left((1+\alpha)I-\alpha J\right)^{-1}\epsilon  \nonumber\\
&~=~\alpha^2~\epsilon^{\top}  \frac{(1+\alpha-K\alpha)I+\alpha J}{(1+\alpha)(1+\alpha-K\alpha)} \epsilon \nonumber\\
&~=~\frac{\alpha^2}{(1+\alpha)(1+\alpha-K\alpha)}~  \left((1+\alpha-K\alpha)\epsilon^{\top}\epsilon+\alpha \epsilon^{\top} J\epsilon\right)  \nonumber
\end{align*}
\begin{align}
&~=~\frac{\alpha^2}{(1+\alpha)(1+\alpha-K\alpha)}~  \left((1+\alpha-K\alpha)K+\alpha (2n-K)^2\right) \label{eqn:cmsq}\\
&~=~ \alpha^2 \frac{4\alpha\cdot n\cdot\left(n-K\right) +\left(1+\alpha\right)K}{(1+\alpha)\left(1+\alpha -K\alpha\right)}, \nonumber
\end{align}
where the second to last equality follows from the two facts 
\[\epsilon^{\top}\epsilon=\sum_{i=1}^K \left(\epsilon^{(i)}\right)^2=K \quad \textrm{and} \quad \epsilon^{\top} J\epsilon=\left(\sum_{i=1}^K \epsilon^{(i)}\right)^2=\left(n-(K-n)\right)^2=(2n-K)^2.\]
Furthermore, in an equiangular set with angle $\alpha$, no two vectors can be orthogonal since $\alpha\neq 0$ (note $0<\alpha<1$ by Definition \ref{def:def1}).  Thus no vector in $X$ can lie in the orthogonal complement $P^{\bot}$ of the $K$-base and $\ell(K, n)=\ip{h}{h} > 0$.
By~\eqref{eqn:cmsq},
\[
\ell(K,n) ~=~ \alpha^2 \frac{4 \alpha(n - K/2 )^2 + K (1 +\alpha - K \alpha)}{(1+\alpha)(1+\alpha-K\alpha)}.
\]
Thus, as a function of $n$, $\ell(K,n)$ is strictly increasing over $1 \leq n \leq   K/2 $ with
\[
\ell(K,K-\frac{(1/\alpha)+1}{2}) ~=~ \alpha\quad \textrm{and} \quad \ell(K, \lfloor K/2\rfloor) ~\leq~  \ell(K, K/2) ~=~ \alpha^2\frac{K}{1+\alpha}~<~1.
\]
So we have proved (\ref{eq:beta}). 
\end{proof}


\begin{definition}\label{def:ec}
For any $x, y\in X\backslash \{p_1, \ldots, p_K\}$, suppose 
 \[x~=~h_1+c_1  \;\text{and}\;y~=~h_2+c_2,\;\text{where}\;\;h_1, h_2\in P \;
\text{and}\;c_1, c_2\in P^{\bot}.\]
We say $x\sim y$ w.r.t.~$\{p_1, \ldots, p_K\}$ if  $h_1$ and $h_2$ are linearly dependent.
The resulting {\em equivalence classes} w.r.t. $\{p_1, \ldots, p_K\}$ are
\[\overline{x}~=~\{y\in X\backslash \{p_1, \ldots, p_K\}|x\sim y\},\;\;\; \forall x\in X\backslash \{p_1, \ldots, p_K\}.\]
\end{definition}

\begin{remark}
Each equivalence class $\ex$  is exactly a ``pillar" of $h$ defined in \citep[Page 501, Section 4]{LS1973}. 
However, with the exception of the fairly trivial \citep[Theorem 4.4]{LS1973}, only the extremal case $\kx=(1/\alpha)+1$ is discussed. Our discussion here is more general for any possible
$\kx$. 
\end{remark}
\begin{lemma}\label{lm:ec}
For any $x, y\in X\backslash \{p_1, \ldots, p_K\}$, suppose $x~=~h_1+c_1$ and $y~=~h_2+c_2$
with
\[ (\langle x, p_1\rangle, ~\ldots, ~\langle x, p_K\rangle)^{\top} ~= ~\alpha\cdot \epsilon_1, \quad \textrm{and} \quad
(\langle y, p_1\rangle, ~\ldots, ~\langle y, p_K\rangle)^{\top} ~=~ \alpha\cdot \epsilon_2,\]
where $h_1, h_2\in P$, $c_1, c_2\in P^{\bot}$ and 
$\epsilon_1, \epsilon_2$ are  $(1, -1)$-vectors in ${\mathbb R}^K$. 
Then the three statements below are equivalent 
\begin{center}
{\em(i)}\;$x\sim y$;\;\;\;\;\;\;\;\;\;\;\;\;\;\;\;\;\;\;\;\;\;\;\;
{\em(ii)}\;$\epsilon_2=\pm \epsilon_1$;\;\;\;\;\;\;\;\;\;\;\;\;\;\;\;\;\;\;\;\;\;\;\;
{\em(iii)}\;$h_1=\pm h_2$.
\end{center}
Hence for each equivalence class $\ex$, there exists an $\epsilon_1$ such that
 \[\ex~=~\{x\in X\backslash \{p_1, \ldots, p_K\}|(\langle x, p_1\rangle, ~\ldots, ~\langle x, p_K\rangle)^{\top} ~= ~\pm \alpha\cdot \epsilon_1\}.\]
 \end{lemma}
  \begin{proof}
 Assume that $K<(1/\alpha)+1$. 
    It follows from \eqref{eq:h} in Proposition~\ref{lm:project} that for $j=1,2$
  \[
  h_j~ = ~\left( \begin{array}{cccc}p_1 & p_2 & \hdots & p_K \end{array}\right)  \alpha~\left((1+\alpha)I-\alpha J\right)^{-1}~\epsilon_j,
  \]
 where the matrix on the left-hand-side of the product $\left( \begin{array}{cccc}p_1 & p_2 & \hdots & p_K \end{array}\right)$  is the $r \times K$ matrix with column vectors $p_1, \ldots, p_K$. Thus,  there exists a real number $\lambda \neq 0$ such that $h_1 = \lambda h_2$ if and only if
 \[
0~= ~h_1 -\lambda h_2 ~ =~ \left( \begin{array}{cccc}p_1 & p_2 & \hdots & p_K \end{array}\right)  \alpha~\left((1+\alpha)I-\alpha J\right)^{-1}~\left(\epsilon_1- \lambda \eps_2\right).\]
If $K=(1/\alpha)+1$, the argument is similar, but additionally makes use of the fact that in this case
\[
\operatorname{kern}\left( \begin{array}{cccc}p_1 & p_2 & \hdots & p_K \end{array}\right)~=~\lspan\left(\begin{array}{cccc} 1 & 1 & \hdots & 1\end{array}\right)^{\top}.
\] 
 \end{proof}
 \begin{corollary}\label{cry:ec}
 Each equivalence classe $\ex$ w.r.t.~$\{p_1, \ldots, p_K\}$ is switching equivalent to an equiangular set $Y=\{y_1, \ldots, y_s\}$ such that 
there exist $h\in P$ and a $(1, -1)$-vector $\epsilon \in {\mathbb R}^{K}$ such that for every $y_i$ $(i=1, \ldots, s)$, we have 
 \[y_i~=~h+d_i,\;\;\;\;\text{where}\;\;d_i\in P^{\bot} \quad \textrm{and} \quad (\langle y_i, p_1\rangle, ~\ldots, ~\langle y_i, p_K\rangle)^{\top} ~=~ \alpha\cdot \epsilon.\]
 \end{corollary}
 \begin{proof}
 Let $\ex=\{x_1, \ldots, x_s\}$. 
For each $i=1, \ldots, s$, $x_i$ can be writen as 
\[x_i=h_i+c_i, \;\;\;\text{where}\;\;h_i\in P \;
\text{and}\;c_i\in P^{\bot}.\] 
By Lemma~\ref{lm:ec}, 
there exists $h\in P$ such that $h_i=h$, or $-h$ for each $i=1, \ldots, s$.
Let 
\[Y~=~\{x_i| x_i = h + c_i, x_i\in\ex\} \cup \{-x_i| x_i = -h + c_i, x_i\in\ex\}.\]
Note that $Y$ is switching equivalent to $\ex$.
Rename the vectors in $Y$ as $y_1, \ldots, y_s$. Then we can write each of them as
\[y_i=h+d_i, \;\;\;\text{where}\;\;h\in P \;
\text{and}\;d_i\in P^{\bot}.\] 
So for any $i$, $(\langle y_i, p_1\rangle, ~\ldots, ~\langle y_i, p_K\rangle)^{\top} ~=~ (\langle h, p_1\rangle, ~\ldots, ~\langle h, p_K\rangle)^{\top}=\alpha\cdot \epsilon$ for a $(1, -1)$-vector $\epsilon$. 
 \end{proof}
 By the definition of equivalence classes w.r.t.~$\{p_1, \ldots, p_K\}$ (Definition \ref{def:ec}), we have $X=\bigcup_{x\in X}\ex$. 
  In order to derive an upper bound of $|X|$, we naturally have two questions below. 
 \begin{itemize}
 \item[]{\bf (Question 1).}~
 How many (necessarily finitely many) equivalence classes are there in $X$?
 \item[]{\bf (Question 2).}~For each equivalence class $\ex\subset X$, what is an upper bound for $|\ex|$? 
 \end{itemize}
 We answer the two questions in the following subsections according to two different cases $\kx=(1/\alpha)+1$ (Theorem \ref{general}) and 
 $\kx<(1/\alpha)+1$ (Theorems~\ref{thm:ec} and~\ref{th:s2}).

\subsubsection{$\kx=(1/\alpha)+1$}\label{reviewseidel}


The extremal case $\kx=(1/\alpha)+1$ was studied in  detail  in  \citep{LS1973}.   In this case, $\{p_1, p_2, \hdots, p_K\}$ forms an equiangular tight frame for its span and can be thought of as vectors pointing to vertices of a regular simplex centered at the origin \citep{FJKM17}. (Recall
$\sum_{j=1}^K p_j =0$ by Proposition \ref{lm:easy}.) 
For this case,  we answer (Question 1) and (Question 2) in Theorem \ref{general}. The proof is given in Appendix~\ref{App:AA}. 
\begin{theorem}\label{general} 
Let $X\subset {\mathbb R}^r$ be an equiangular set with angle $\alpha$, 
where $1/\alpha$ is an odd number greater or equal to $3$. 
 If $K=\kx=(1/\alpha)+1$, 
 then there are $\frac{1}{2}\binom{K}{K/2}$ equivalence classes $\ex$ w.r.t.~any fixed $K$-base, and for each
 $\ex$, 
 \[|\ex|~\leq~ r -  K + 1 +  \lfloor 2\alpha\frac{r-K+1}{1-\alpha} \rfloor.\]
  Hence, 
\[|X|~\leq~ K + \frac{1}{2}\binom{K}{K/2}\left(r -  K + 1 +  \lfloor 2\alpha\frac{r-K+1}{1-\alpha} \rfloor \right).\]
\end{theorem}

For $\alpha=1/3$ and $1/5$, the upper bound in Theorem~\ref{general} can be reduced significantly by applying spectral graph theory. For instance, if $\alpha=1/3$, $|X| \leq 2(r-1)$ for $r \geq 15$, and any set $X$ which attains this upper bound must have $\kx=(1/\alpha)+1=4$ \citep[Theorem 4.5]{LS1973}. 
What is deeply hidden in its proof is that the only connected simple graph with maximum eigenvalue $1$ is the complete graph on two vertices. Theorem~\ref{ls1973} below is proved by the fact that the connected simple graphs  with maximum eigenvalue $2$ only have $5$ patterns \cite[Theorem 5.1]{LS1973}.   

\begin{theorem}\label{ls1973}\cite[Theorem 5.7]{LS1973}
Let $X\subset {\mathbb R}^r$ be an equiangular set with angle $1/5$. 
 If $K_{1/5}(X)=6$, then 

\[|X|~\leq~
\begin{cases}
 276, &  23\leq r\leq 185 \\ 
 r+1+\lfloor \left(r-5\right)/2\rfloor, & r\geq 185.
\end{cases}
\]
\end{theorem}

\subsubsection{$\kx < (1/\alpha)+1$}\label{littlek}
Let $X\subset {\mathbb R}^r$ be an equiangular set with angle $\alpha$. 
In this subsection, we answer (Question 1) and (Question 2) for the case $K=\kx < (1/\alpha)+1$. 
When $\alpha$  and a $K$-base $\{p_1, \ldots, p_K\}$ are fixed, we notice that for any  vector $x\in X\backslash \{p_1, \ldots, p_K\}$,  the norm of its projection $h$ onto the subspace $P$ (spanned by $p_1, \ldots, p_K$), that is the value of $\ell(K, n)$ in (\ref{eq:l}), only
depends on $n$, namely the number of positive inner products among $\langle x, p_1\rangle, ~\ldots, ~\langle x, p_K\rangle$. We further note that the function $\ell(K, n)$  is symmetric w.r.t.~$n$ and $K-n$. That means when we have $n$ or $K-n$ positive signs among $\langle x, p_1\rangle, ~\ldots, ~\langle x, p_K\rangle$, the norm 
of $h$ will be the same.  
This should be expected, since if $X$ is equiangular, $(X\backslash \{x\}) \cup \{-x\}$ is also equiangular, and these two equiangular sets are 
switching equivalent. Inspired by these observations,  
we define subsets $X(K, n)$ of $X\backslash \{p_1, \ldots, p_K\}$  for $n=0, \ldots, \lfloor K/2\rfloor$, 
\begin{equation}\label{eq:x}
{
X(K, n)~\triangleq~\{x\in  X\backslash \{p_1, \ldots, p_K\}|\exists~\text{exactly}~n~\text{or}~K-n~\text{positive signs among}~\langle x, p_1\rangle, \ldots, \langle x, p_K\rangle\}.
}
\end{equation}
Then by the formula $(\ref{eq:l})$, for any two distinct vectors $x, y\in X(K, n)$, if we project them onto the subspace $P$, the norms of their projections are the same.  We next  show by
 Proposition \ref{pp:simplify} that  $X(K, 0)=\emptyset$.  
 \begin{proposition}\label{pp:simplify}
 Let $X\subset {\mathbb R}^r$ be an equiangular set with angle $\alpha$. If $K=\kx< (1/\alpha)+1$, 
then $X(K, 0)=\emptyset$.
\end{proposition}
\begin{proof}
Assume that $X(K, 0)\neq \emptyset$.  For any $x\in X(K, 0)$, by the definition of $X(K, 0)$ in (\ref{eq:x}), 
$\langle x, p_1\rangle, \ldots, \langle x, p_K\rangle$ are either all $-\alpha$ or all $\alpha$. That means we have 
\[
G(x, p_1, \ldots, p_K)~=~(1+\alpha) I -\alpha J \;\; \text{or}\;\; G(-x, p_1, \ldots, p_K)~=~(1+\alpha) I -\alpha J.
\]
By Lemma \ref{lm:switch}, $G(x, p_1, \ldots, p_K)\cong(1+\alpha) I -\alpha J$.
So by the definition of 
$\kk(X)$ in $(\ref{eq:k})$, 
$K+1\leq \kk(X)=K$. This is a contradiction. 
\end{proof}
Since  $X(K, 0)=\emptyset$, we write $X$ as a disjoint union 
\begin{equation}\label{eq:decompose}
X~=~\{p_1, \ldots, p_K\} \bigcup (X\backslash \{p_1, \ldots, p_K\})
~=~\{p_1, \ldots, p_K\} \bigcup\bigcup\limits_{n=1}^{\lfloor K/2\rfloor
} X(K, n).
\end{equation}
Below, in Theorem \ref{thm:ec}, 
we give the number of equivalence classes $\ex$ in  $X(K, n)$ for each $n$, and 
 in Theorem \ref{th:s2}, for each equivalence class $\ex$ in $X(K, n)$, we give an upper bound on $|\ex|$ in terms of $r, \alpha, K, \ell(K, n)$.  
  The proof of Theorem \ref{th:s2} is given in Appendix~\ref{App: AB}. 

 \begin{theorem}\label{thm:ec}
 Let $X\subset {\mathbb R}^r$ be an equiangular set with angle $\alpha$. 
 Suppose $K=\kx<(1/\alpha)+1$. 
 Fix a $K$-base $\{p_1, \ldots, p_K\}$ and define subsets $X(K, n)$ for $n=1, \ldots, \lfloor K/2\rfloor$ as in (\ref{eq:x}).   
 For each $n$, if $2n <K$, then
 the number of equivalence classes $\ex$ in $X(K, n)$ w.r.t $\{p_1, \ldots, p_K\}$ is $\binom{K}{n}$, and if $2n=K$, 
 then  the number of equivalence classes $\ex$ is $\frac{1}{2}\binom{K}{n}$.
 \end{theorem}
 \begin{proof}
$X(K,n)$ is defined in (\ref{eq:x}) by the set of $x \in X \backslash \{p_1, \hdots, p_K\}$ with $n$ or $K-n$ positive inner products among  $\ip{x}{p_1}, \hdots \ip{x}{p_K}$, whereas it follows from Lemma~\ref{lm:ec} that the equivalence classes $\overline{x}$ are determined by \emph{which} specific inner products are positive.  Thus the theorem follows from a simple combinatorial argument.  The only slight trick is that when $2n=K$, one notes that when $x \in \overline{x} \in X(2n,n)$, $-x \in \overline{x}$ as well. 
\end{proof}
 
 \begin{theorem}\label{th:s2}
  Let $X\subset {\mathbb R}^r$ be an equiangular set with angle $\alpha$. 
 Suppose $K=\kx<(1/\alpha)+1$. 
 Fix a $K$-base $\{p_1, \ldots, p_K\}$ and define subsets $X(K, n)$ for $n=1, \ldots, \lfloor K/2\rfloor$ as in (\ref{eq:x}).
 For $n=1, \ldots,  \lfloor K/2\rfloor$ and for each equivalence class  $\ex\subset X(K, n)$ w.r.t $\{p_1, \ldots, p_K\}$, we have the following upper bounds on $|\ex|$. 
\begin{itemize}
\item[{\em (1)}] If $n=1$,  then
\begin{equation*}\label{bound1}
|\ex|~\leq~ 
\begin{cases}
r-K,  & 1\geq K - \frac{(1/\alpha)+1}{2} 
\\
\frac{1-\alpha}{\ell(K, 1)-\alpha},
& 1<K - \frac{(1/\alpha)+1}{2}.
\end{cases}
\end{equation*}
\item[{\em (2)}] If $1<n< K- \frac{(1/\alpha)+1}{2}$, then
\begin{equation*}\label{bound2}
|\ex|~\leq~ r+1.
\end{equation*}

\item[{\em (3)}] If $n= K- \frac{(1/\alpha)+1}{2}$, then
\begin{equation*}\label{bound3}
|\ex|~\leq~   r- K   +  \lfloor 2\alpha\frac{r-K }{1-\alpha} \rfloor. 
\end{equation*}

\item[{\em (4)}] If  $K- \frac{(1/\alpha)+1}{2} < n < \lfloor\frac{K}{2}\rfloor$, then
\begin{equation*}\label{bound4}
|\ex|~\leq~ 
 s\left(r, \beta, \gamma \right), \quad\textrm{where}\enskip  \beta=\frac{\alpha-\ell(K, n)}{1-\ell(K, n)} \enskip \textrm{and} \enskip\gamma=\frac{-\alpha-\ell(K, n)}{1-\ell(K, n)}.
\end{equation*}
\end{itemize}
\end{theorem}

\begin{corollary}\label{cry:k23}
Let $X\subset {\mathbb R}^r$ be an equiangular set with angle $\alpha$.  
\begin{itemize}
\item[{\em (1)}] If $\kx=2<(1/\alpha)+1$, then
$|X|~\leq~r$. 
\item[{\em (2)}] If $\kx=3<(1/\alpha)+1$, then
$|X|~\leq~3r - 6$.
\end{itemize}
\end{corollary}
\begin{proof}
Let  $K=\kx$. If $K=2$,  then the partition (\ref{eq:decompose}) becomes 
$X
=\{p_1, p_2\} \bigcup  X(2, 1)$. By Theorem \ref{thm:ec}, the number of equivalence classes $\ex$ in 
$X(2, 1)$ is $\frac{1}{2}\binom{2}{1}=1$. 
By Theorem \ref{th:s2} (1), 
for any $x\in X(2, 1)$, $|\ex|\leq r-K= r - 2$. So we have 
$|X|
=2 + |X(2, 1)|\leq 2+ r-2 =r.$

If $K=3$, then  the partition (\ref{eq:decompose}) becomes 
$X
=\{p_1, p_2, p_3\} \bigcup  X(3, 1)$.
By Theorem \ref{thm:ec}, the number of equivalence classes $\ex$ in 
$X(3, 1)$ is $\binom{3}{1}=3$. 
By Theorem \ref{th:s2} (1), for any $x\in X(3, 1)$, $|\ex|\leq r-K=r-3$. 
So $|X| = 3 + |X(3, 1)| \leq 3+ 3\times (r-3) =3r-6$.
\end{proof}

\subsection{A procedure: piecing together the results}\label{pipeline}

Once we fix a dimension $r$, the basic procedure to determine an upper bound on the size of an equiangular set $X$ in $\bR^r$ is as follows:
\begin{itemize}
\item From Proposition~\ref{angles}, we know that any maximal set of equiangular lines will either have size $2r+3$ or have angle $\alpha = 1/(2L-1)$, where $2\leq L\leq (1+\sqrt{2r})/2$. This gives us finitely many angles to test.
\item For $\alpha = 1/3$ and $r \geq 15$, we may apply~\citep[Theorem 4.5]{LS1973} to determine that $\abs{X} \leq 2(r-1)$.
\item For $\alpha < 1/3$, we consider the possible sizes of a $K$-base in such an $X$, which will be by Propositions~\ref{lm:easy} and~\ref{lm:easy2} between $2$ and $1 + 1/\alpha$.
\item Then for the $K$-base size $1 + 1/\alpha$, we apply Theorem \ref{general} to derive an upper bound, and for each other possible $K$-base size ($<1 + 1/\alpha$), we consider the partition (\ref{eq:decompose}) of the remaining elements of $X$ based on the number ($n$ or $K-n$) of positive inner products with the $K$-base.  By Theorem~\ref{thm:ec}, there will be $\lfloor  K/2\rfloor$ such sets of size at most $\binom{K}{n}$ or $\frac{1}{2}\binom{K}{n}$.
\item By Lemma~\ref{lm:ec}, we then split the partition further into equivalence classes based on \emph{with which} $K$-base elements an element has positive inner products.  We bound the size of each of these equivalence classes by the bounds in Theorem~\ref{th:s2}.
\end{itemize}
Below, we illustrate one example to show how to apply  the above procedure  to compute an upper bound of an equiangular set $X\subset {\mathbb R}^r$  for a particular $r$ and $\alpha$.  This approach will in particular be used to created Figure~\ref{fig:bound7} (Section \ref{experiment}).
\begin{example}\label{ex:rb7}
Suppose $r = 236$ and $\alpha = 1/7$. 
By Propositions~\ref{lm:easy} and~ \ref{lm:easy2},  $K=\kx$ is at most $1+1/\alpha=8$ and at least $2$. Below, we compute an upper bound of $|X|$ for the cases
$K=2, 3, 4, 5, 6, 7$, and  $8$. Overall, the maximum upper bound in these cases is $15673$ and occurs when $K=7$.  
\begin{itemize} 
\item By Corollary \ref{cry:k23}, if $K=2$, then 
$|X|
\leq r  = 236$, and
if $K=3$, then
$|X|
\leq 3r - 6 =702$.
\item If $K=4$, then the partition (\ref{eq:decompose}) becomes $X=\{p_1, p_2, p_3, p_4\}\cup \cup_{n=1}^2 X(4, n)$.
For $n=1$, 
by Theorem \ref{thm:ec}, the number of equivalence classes $\ex$ in 
$X(4, 1)$ is $\binom{4}{1}=4$. 
By  Theorem 
\ref{th:s2} (1), for any $x\in X(4, 1)$, \[|\ex|~\leq~ r-K~=~ 232.\]
For $n=2$, 
by Theorem \ref{thm:ec}, the number of equivalence classes $\ex$ in 
$X(4, 2)$ is $\frac{1}{2}\binom{4}{2}=3$. 
Since $K- ((1/\alpha)+1)/2 =4-(5+1)/2= 1< n = 2$, we will apply the bound in Theorem~\ref{th:s2} (4).  To this end, 
by (\ref{eq:l}), 
we calculate $\ell(4,2) =1/14$, which implies $\beta =(1/7-\ell(4, 2))/(1-\ell(4, 2))=  1/13$, and $\gamma =(-1/7-\ell(4, 2))/(1-\ell(4, 2))= -3/13$,
and for any $x\in X(4, 2)$, 
\[|\ex|~\leq~  s\left(r, 1/13, -3/13 \right) ~\leq~ 1832,\]
where the upper bound $1832$ for $ s\left(r, 1/13, -3/13 \right)$ is computed by Theorem \ref{sdp} (running SDP). 
Thus, by the partition, 
\[|X| 
~=~
4 + |X(4, 1)| + |X(4, 2)|
~\leq~ 4 + 4\cdot 232 +  3\cdot 1832
~=~ 6428.\]
 \item If $K=5$,
then the partition (\ref{eq:decompose}) becomes $X=\{p_1, p_2, p_3, p_4, p_5\}\cup \cup_{n=1}^2 X(5, n)$.
For $n=1$, by Theorem \ref{thm:ec}, the number of equivalence classes $\ex$ in 
$X(5, 1)$ is $\binom{5}{1}=5$. 
By 
Theorem 
\ref{th:s2} (1), for any $x\in X(5, 1)$, \[|\ex|~\leq~ r-K~=~ 231.\]
For $n=2$, by Theorem \ref{thm:ec}, the number of equivalence classes $\ex$ in 
$X(5, 2)$ is $\binom{5}{2}=10$. 
 By  
Theorem 
\ref{th:s2} (4), for any $x\in X(5, 2)$, 
\[|\ex|~\leq~   s\left(r, 1/19, -5/19\right)~ \leq ~935.\]
So by the partition, 
\[|X|
~=~
5 + |X(5, 1)| + |X(5, 2)|
~\leq~ 5 + 5\cdot 231 +  10\cdot 935 ~=~ 10510.
\]
\item If $K=6$, 
then the partition (\ref{eq:decompose}) becomes $X=\{p_1, p_2, p_3, p_4, p_5, p_6\}\cup \cup_{n=1}^3 X(6, n)$.
For $n=1$, 
by Theorem \ref{thm:ec}, the number of equivalence classes $\ex$ in 
$X(6, 1)$ is $\binom{6}{1}=6$. 
By  
Theorem 
\ref{th:s2} (1), for any $x\in X(6, 1)$,
\[|\ex|~\leq~ \frac{1-\alpha}{\ell(K, 1)-\alpha} ~=~ \frac{1-1/7}{\ell(6, 1)-1/7}~=~ 8.\]
For $n=2$, 
by Theorem \ref{thm:ec}, the number of equivalence classes $\ex$ in 
$X(6, 2)$ is $\binom{6}{2}=15$. 

By Theorem 
\ref{th:s2} (3), for any $x\in X(6, 2)$, 
\[|\ex|~\leq ~ r- K  +  \lfloor 2\alpha\frac{r-K }{1-\alpha} \rfloor ~=~ 236 -  6 +  \lfloor 2\cdot \frac{1}{7}\cdot \frac{236 - 6}{1-1/7} \rfloor~ = ~306.\]
For $n=3$, 
by Theorem \ref{thm:ec}, the number of equivalence classes $\ex$ in 
$X(6, 3)$ is $\frac{1}{2}\binom{6}{3}=10$. 
By  Theorem 
\ref{th:s2} (4), for any $x\in X(6, 3)$, 
\[|\ex|~\leq ~  s\left(r, 1/25, -7/25\right) ~\leq~ 635.\]
So by the partition,
\[|X|
~=~
6 + |X(6, 1)| + |X(6, 2)| + |X(6, 3)|
~\leq~ 6 + 6\cdot 8 +  15\cdot 306 + 10\cdot 635~ =~ 10994.
\]
 \item If $K=7$, then 
then the partition (\ref{eq:decompose}) becomes $X=\{p_1, p_2, p_3, p_4, p_5, p_6, p_7\}\cup \cup_{n=1}^3 X(7, n)$.
For $n=1$, 
by Theorem \ref{thm:ec}, the number of equivalence classes $\ex$ in 
$X(7, 1)$ is $\binom{7}{1}=7$. 
By Theorem 
\ref{th:s2} (1), for any $x\in X(7, 1)$, 
\[|\ex|~\leq~ \frac{1-\alpha}{\ell(K, 1)-\alpha}~ = ~ \frac{1-1/7}{\ell(7, 1)-1/7}~=~ 2.\]
For $n=2$,  
by Theorem \ref{thm:ec}, the number of equivalence classes $\ex$ in 
$X(7, 2)$ is $\binom{7}{2}=21$.
By  Theorem 
\ref{th:s2} (2), for any $x\in X(7, 2)$, $|\ex|\leq r+1 = 237$.
For $n=3$, 
by Theorem \ref{thm:ec}, the number of equivalence classes $\ex$ in 
$X(7, 3)$ is $\binom{7}{3}=35$. 
By Theorem 
\ref{th:s2} (3), for any $x\in X(7, 3)$, 
\[|\ex|~\leq~  r- K   +  \lfloor 2\alpha\frac{r-K }{1-\alpha} \rfloor ~=~ 236 -  7 +  \lfloor 2\cdot \frac{1}{7}\cdot \frac{236 - 7}{1-1/7} \rfloor ~=~ 305.\]
So by the partition,
\[|X|
~=~
7 + |X(7, 1)| + |X(7, 2)| + |X(7, 3)|
~\leq~ 7 + 7\cdot 2 +  21\cdot 237 + 35\cdot 305 ~=~  15673.
\]
\item If $K=8$, then by Theorem \ref{general}, 
\begin{align*}
|X| &~ \leq~ K + \frac{1}{2}\binom{K}{K/2}\left(r -  K + 1 +  \lfloor 2\alpha\frac{r-K+1}{1-\alpha} \rfloor \right) \\
     &~ = ~8 + \frac{1}{2}\binom{8}{4}\left(236 -  8 + 1 +  \lfloor 2\cdot \frac{1}{7}\cdot \frac{236 - 8 + 1}{1-1/7} \rfloor \right)  ~= ~10683.
\end{align*}
\end{itemize}
\end{example}

\section{Relative bound for $\alpha = 1/5$}\label{rb5}
Applying the procedure in the previous section as presented in Example \ref{ex:rb7}, one can derive an upper bound for equiangular sets with any angle. But this upper bound might not be optimal. One reason is that when two different equivalence classes are not empty, each equivalence class might not really be filled with the number of vectors given in Theorem \ref{general} or Theorem \ref{th:s2} (evidence of this can be seen in \citep[Theorem 4.3]{LS1973}).
In this section, we use this observation to prove an upper bound when $\alpha=1/5$ (Theorem \ref{thm:rb5}). The basic idea is to apply the procedure but then to analyze further when two equivalence classes are not empty (see the proofs of Lemmas \ref{lm:rb4}--\ref{lm:rb5} in Appendices~\ref{app:AC}--\ref{app:AD}).  One could perform similar analysis for any smaller angles (such as $\alpha=1/7, 1/9, \ldots$), but it would be a very technical task. 



\begin{lem}\label{lem:rv}
For any $r\in \bN$ and $0 \leq \alpha < 1$, there exist at least $r$ vectors in $\bR^r$ with pairwise inner product $\alpha$. 
\end{lem}
\begin{proof}
Consider the $r\times r$ real symmetric matrix $G=(1-\alpha)I_r + \alpha J_r$. By Lemma \ref{pp:basic}, $G$ has two  distinct eigenvalues 
\[\lambda_1=1+(r-1)\alpha,\;\;\;\lambda_2=1-\alpha.\]
Both eigenvalues are positive. So $G$ is positive definite. By the well-known Cholesky decomposition, there exists a unique $r\times r$ lower triangular matrix 
$L$ such that $G = LL^{T}$. Suppose the column vectors of $L$ are $x_1, \ldots, x_r$. Then $G$ is the Gramian matrix of $x_1, \ldots, x_r$.
\end{proof}
Geometrically, we may think of ``pushing'' vectors in an orthonormal basis for $\bR^r$ towards each other until the desired angle is achieved. 
\begin{corollary}\label{lm:rb4lower}
For any positive integer $r$,  
if $s\left(r, 1/13, -5/13 \right)$ denotes the maximum cardinality of a spherical two-distance set in ${\mathbb R}^r$ w.r.t.~$1/13$ and $-5/13$, then
\[s\left(r, 1/13, -5/13 \right)~\geq ~r.\]
\end{corollary}
\begin{proof}
We know from Lemma~\ref{lem:rv} that there are at least $r$ vectors with inner product $1/13$.
\end{proof}

For the cases $\km(X)=4, 5$, we present Lemmas \ref{lm:rb4}--\ref{lm:rb5} below. The proofs of the two lemmas are given in Appendices~\ref{app:AC} and~\ref{app:AD}, respectively. 
\begin{lemma}\label{lm:rb4}
Let $X\subset {\mathbb R}^r$ be an equiangular set with angle $1/5$.   If $\km(X)=4$, then
\[|X|~\leq~148  +  3\cdot s\left(r, 1/13, -5/13 \right).\]
\end{lemma}

\begin{lemma}\label{lm:rb5}
Let $X\subset {\mathbb R}^r$ be an equiangular set with angle $1/5$.    If $\km(X)=5$, then
\[|X|~\leq~
\begin{cases}
 290, &  23\leq r\leq 185 \\ 
 r+15 +\lfloor (r-5)/2\rfloor,& r\geq 185.
\end{cases}
\]
\end{lemma}


\begin{theorem}\label{thm:rb5}
Let $X\subset {\mathbb R}^r$ be an equiangular set with angle $1/5$.  
If $r>60$, then
 \[|X|~\leq ~148 +  3\cdot s\left(r, 1/13, -5/13 \right).\] 
\end{theorem}
\begin{proof}
By Proposition \ref{lm:easy},  $\km(X)$ is at most $(1/5)^{-1}+1=6$, and by 
Proposition~\ref{lm:easy2}, $\km(X)$ is at least $2$. So it is only possible for $\km(X)$ to be 
$2, 3, 4, 5$, or  $6$.
\begin{itemize}
\item By Corollary \ref{cry:k23}, if $\km(X)=2$, then $|X|\leq r$, and if $\km(X)=3$, then $|X|\leq 3r-6$.
\item If $\km(X)=4$, then by Lemma \ref{lm:rb4}, $|X|\leq ~148 +  3\cdot s\left(r, 1/13, -5/13 \right)$. 
\item If $\km(X)=5$, then by Lemma \ref{lm:rb5}, $|X|\leq\max \left(290,  r+15 + \lfloor(r-5)/2\rfloor\right)$.
\item If $\km(X)=6$, then by Theorem \ref{ls1973}, $|X|\leq\max \left(276,  r+1 + \lfloor(r-5)/2\rfloor\right)$.
\end{itemize}
Overall, $|X|$ should be bounded by the maximum upper bound in the above cases. Note $r\leq 3r-6$ when $r\geq 3$. 
So 
\begin{align*}
|X|&~\leq~ \max\left(r, ~3r-6, ~148+ 3\cdot s\left(r, 1/13, -5/13\right), ~290,  ~r+15 + \lfloor (r-5)/2\rfloor, ~276,  ~r+1 + \lfloor (r-5)/2\rfloor\right)\\
    &~=~ \max\left(~3r-6,  ~148+ 3\cdot s\left(r, 1/13, -5/13\right), ~290,  ~r+15 + \lfloor (r-5)/2\rfloor\right).
\end{align*}
By Corollary~\ref{lm:rb4lower}, $s\left(r, 1/13, -5/13\right)\geq r$. So
\[148+ 3\cdot~s\left(r, 1/13, -5/13\right)~\geq~ 148+3r~ > ~\max\left\{3r-6,r+15 + \lfloor (r-5)/2\rfloor\right\}.\]
and hence 
\[|X|~\leq~\max \left(148+ 3\cdot s\left(r, 1/13, -5/13\right), ~~290\right).\] 
If $r>60$, then we further have
\[148+ 3\cdot s\left(r, 1/13, -5/13\right)~\geq~148 + 3r~> ~148 + 3\times 60 ~>~290.\]
So if $r>60$, 
\[|X|~\leq~148+ 3\cdot s\left(r, 1/13, -5/13\right).\]
\end{proof}
\begin{lemma}\label{lm:gy2016}\cite[Corollary 4]{GY2016}
If  $Y$ is a spherical two-distance set in ${\mathbb R}^r$ w.r.t.~$\alpha$ and $\beta$, then
\[|Y|~\leq~\frac{r+2}{1-\frac{r-1}{r(1-\alpha)(1-\beta)}}.\]
\end{lemma}

\begin{corollary}\label{cry:rb5}
For a given equiangular set $X\subset {\mathbb R}^r$ with angle $1/5$ if $r>60$, then
\[|X|~\leq ~148 +  \frac{648r(r+2)}{47r+169}.\] 
\end{corollary}
\begin{proof}
By Theorem \ref{thm:rb5} and Lemma \ref{lm:gy2016} \cite[Corollary 4]{GY2016}, 
\[|X|~\leq ~148 +  3\cdot s\left(r, 1/13, -5/13 \right)~\leq~ 148+3\frac{r+2}{1-\frac{r-1}{r(1-1/13)(1+5/13)}} ~=~  148+\frac{648r(r+2)}{47r+169}.\]
\end{proof}

\begin{remark}
We have a few remarks about Theorem \ref{thm:rb5} and Corollary \ref{cry:rb5}.
\begin{itemize}
\item[(i)] Note $s\left(r, 1/13, -5/13 \right)$ in the bound given in Theorem \ref{thm:rb5} can be computed bounded by Theorem \ref{sdp} (running SDP tool). 
Comparing this to the explicit bounds given in Corollary \ref{cry:rb5},  the bounds computed by SDP according to Theorem \ref{thm:rb5} are smaller  for $61\leq r\leq 132$, see the Table \ref{table:rb5}. However, for $r>132$,  Corollary \ref{cry:rb5} gives the smaller bounds.   
{\tiny
\begin{longtable}{|c|c|c|c|c|c|c|c|}\hline
\caption{Compare relative bounds in Theorem \ref{thm:rb5} and Corollary \ref{cry:rb5} $(61\leq r\leq 132)$}\label{table:rb5}\\\hline
dimension & Theorem \ref{thm:rb5} & Corollary \ref{cry:rb5}&dimension & Theorem \ref{thm:rb5} & Corollary \ref{cry:rb5}\\\hline
$r$ & $148+ 3\cdot{\mathrm s}\left(r, \frac{1}{13}, -\frac{5}{13}\right)$ & $148 +  \frac{648r(r+2)}{47r+169}$ &$r$ &$148+ 3\cdot{\mathrm s}\left(r, \frac{1}{13}, -\frac{5}{13}\right)$& $148 +  \frac{648r(r+2)}{47r+169}$ \\\hline
  $ 61$ &$  586$&$  968$&  $ 62$&$  595$&$  982$\\\hline
  $ 63$ &$  604$&$  995$&  $ 64$&$  616$&$ 1009$\\\hline
  $ 65$ &$  625$&$ 1023$&  $ 66$&$  637$&$ 1037$\\\hline
  $ 67$ &$  646$&$ 1050$&  $ 68$&$  658$&$ 1064$\\\hline
  $ 69$ &$  667$&$ 1078$&  $ 70$&$  679$&$ 1092$\\\hline
  $ 71$ &$  691$&$ 1105$&  $ 72$&$  703$&$ 1119$\\\hline
  $ 73$ &$  715$&$ 1133$&  $ 74$&$  727$&$ 1147$\\\hline
  $ 75$ &$  739$&$ 1161$&  $ 76$&$  751$&$ 1174$\\\hline
  $ 77$ &$  763$&$ 1188$&  $ 78$&$  775$&$ 1202$\\\hline
  $ 79$ &$  787$&$ 1216$&  $ 80$&$  799$&$ 1229$\\\hline
  $ 81$ &$  814$&$ 1243$&  $ 82$&$  826$&$ 1257$\\\hline
  $ 83$ &$  841$&$ 1271$&  $ 84$&$  853$&$ 1285$\\\hline
  $ 85$ &$  868$&$ 1298$&  $ 86$&$  883$&$ 1312$\\\hline
  $ 87$ &$  898$&$ 1326$&  $ 88$&$  910$&$ 1340$\\\hline
  $ 89$ &$  925$&$ 1353$&  $ 90$&$  943$&$ 1367$\\\hline
  $ 91$ &$  958$&$ 1381$&  $ 92$&$  973$&$ 1395$\\\hline
  $ 93$ &$  988$&$ 1409$&  $ 94$&$ 1006$&$ 1422$\\\hline
  $ 95$ &$ 1021$&$ 1436$&  $ 96$&$ 1039$&$ 1450$\\\hline
  $ 97$ &$ 1057$&$ 1464$&  $ 98$&$ 1072$&$ 1477$\\\hline
  $ 99$ &$ 1090$&$ 1491$&  $100$&$ 1108$&$ 1505$\\\hline
  $101$ &$ 1126$&$ 1519$&  $102$&$ 1147$&$ 1533$\\\hline
  $103$ &$ 1165$&$ 1546$&  $104$&$ 1186$&$ 1560$\\\hline
  $105$ &$ 1204$&$ 1574$&  $106$&$ 1225$&$ 1588$\\\hline
  $107$ &$ 1246$&$ 1601$&  $108$&$ 1267$&$ 1615$\\\hline
  $109$ &$ 1288$&$ 1629$&  $110$&$ 1309$&$ 1643$\\\hline
  $111$ &$ 1333$&$ 1657$&  $112$&$ 1354$&$ 1670$\\\hline
  $113$ &$ 1378$&$ 1684$&  $114$&$ 1402$&$ 1698$\\\hline
  $115$ &$ 1426$&$ 1712$&  $116$&$ 1453$&$ 1725$\\\hline
  $117$ &$ 1477$&$ 1739$&  $118$&$ 1504$&$ 1753$\\\hline
  $119$ &$ 1531$&$ 1767$&  $120$&$ 1558$&$ 1781$\\\hline
  $121$ &$ 1585$&$ 1794$&  $122$&$ 1612$&$ 1808$\\\hline
  $123$ &$ 1642$&$ 1822$&  $124$&$ 1672$&$ 1836$\\\hline
  $125$ &$ 1702$&$ 1850$&  $126$&$ 1732$&$ 1863$\\\hline
  $127$ &$ 1765$&$ 1877$&  $128$&$ 1798$&$ 1891$\\\hline
  $129$ &$ 1831$&$ 1905$&  $130$&$ 1867$&$ 1918$\\\hline
  $131$ &$ 1900$&$ 1932$&  $132$&$ 1936$&$ 1946$\\\hline
\end{longtable}

}
\pagebreak
\item[(ii)] Notice also that the relative bound given in Corollary \ref{cry:rb5} is asymptotically $\frac{648}{47}r\sim 13.8r$. 
 For any $r>60$, this upper bound is smaller and thus better than \cite[Theorem 3]{GY2016}
\[r(\frac{2}{3}a^2+\frac{4}{7}) + 2 \bigg\vert_{a=5}~=\frac{362}{21}r + 2\sim 17.2r.\]
\item[(iii)] In \citep{neumaier1989}, it is proven that there exists a large integer $N$ such that for any $r\geq N$, the relative bound for 
$\alpha=1/5$ is 
\[r+1 +\lfloor (r-5)/2\rfloor\sim 1.5r.\]
The $N$ stated in \citep[p.~155]{neumaier1989} is claimed to be between $2486$ and $45374$ without proof. 
However, there is still a big gap between the above $1.5r$ and the $13.8r$ in Corollary \ref{cry:rb5}. It would be interesting to shorten the gap 
for $61<r<N$.
\end{itemize}
\end{remark}





\section{Computational Result and Conjecture}\label{main}
We are now prepared to show upper bounds for equiangular sets for $44\leq r\leq 400$. 
\begin{theorem}\label{thm:main}
For $44\leq r\leq 400$, let $m$ be the largest positive integer such that $(2m+1)^2\leq r + 2$. Then an 
upper bound of the maximum number of equiangular lines in ${\mathbb R}^r$ is
\begin{equation}\label{maineq}
\begin{cases}
\frac{4r\left(m+1\right)\left(m+2\right)}{\left(2m+3\right)^2-r}, & r=44, 45, 46, 76, 77, 78, 117, 118, 166, 222, 286, 358\\
\\
\frac{\left(\left(2m+1\right)^2-2\right)\left(\left(2m+1\right)^2-1\right)}{2},&  \text{other}~r~\text{between}~44~\text{and}~400, 
\end{cases}
\end{equation}
and if the upper bound in (\ref{maineq}) can be attained, then the relative angle is 
\begin{equation}\label{mainan}
\begin{cases}
\frac{1}{2m+3}, & r=44, 45, 46, 76, 77, 78, 117, 118, 166, 222, 286, 358\\
\\
\frac{1}{2m+1},&  \text{other}~r~\text{between}~44~\text{and}~400. 
\end{cases}
\end{equation}
\end{theorem}

\begin{proof}
For any $44\leq r\leq 400$, suppose $L$ is the largest positive integer such that $(2L- 1)^2\leq 2r$. For each $r$, by Proposition~\ref{angles}, we compute the relative bounds for each $\alpha = 1/5, 1/7, \ldots, 1/(2L-1)$ and then pick up the maximum.
We summarize the computational results in Table~\ref{proof} (see the complete computational results online \citep{link2018}) 
\begin{table}[h]
\caption{Upper bound of equiangular lines in ${\mathbb R}^r$ for $44\leq r\leq 400$}\label{proof}
\begin{center}
{
\begin{tabular}{c||ccccccc}
$r$&             $44$ & $45$ & $46$  & $47$--$75$ & $76$ & $77$   \\
$|X|\leq$& $422$ & $540$ & $736$ & $1128$&$1216$& $1540$  \\
$\mathrm{arg}\underset{\alpha}\max~|X|$& $\frac{1}{7}$ & $\frac{1}{7}$ & $\frac{1}{7}$ & $\frac{1}{7}$&$\frac{1}{9}$& $\frac{1}{9}$  \\
\hline
\hline
$r$& $78$ & $79$--$116$ & $117$ & $118$&$119$--$165$ & $166$ \\
$|X|\leq$& $2080$ & $3160$ & $3510$ & $4720$& $7140$  & $9296$  \\
$\mathrm{arg}\underset{\alpha}\max~|X|$& $\frac{1}{9}$ & $\frac{1}{9}$ & $\frac{1}{11}$ & $\frac{1}{11}$&$\frac{1}{11}$& $\frac{1}{13}$  \\
\hline
\hline
$r$&$167$--$221$ & $222$ &$223$--$285$ & $286$ & $287$--$357$ & $358$ & $359$--$400$\\
$|X|\leq$ &$14028$ & $16576$& $24976$ &  $27456$&  $41328$  & $42960$ & $64620$\\
$\mathrm{arg}\underset{\alpha}\max~|X|$& $\frac{1}{13}$ & $\frac{1}{15}$ & $\frac{1}{15}$ & $\frac{1}{17}$&$\frac{1}{17}$& $\frac{1}{19}$ &    $\frac{1}{19}$
\end{tabular}
}
\end{center}
\end{table}

\noindent
More specifically, for $\alpha=1/5$, we compute three
upper bounds by Thoerem~\ref{sdp}, Theorem~\ref{thm:rb5}, and Corollary \ref{cry:rb5}, respectively, and then we pick the smaller one of these three upper bounds. 
For $\alpha=1/7$, we compute 
two upper bounds by Theorem~\ref{sdp} and the procedure in Section \ref{pipeline} (Example \ref{ex:rb7}), respectively, and then we pick the smaller one of these two upper bounds. 
For $\alpha = 1/9, \ldots, 1/(2L-1)$,  we compute the SDP bound
according to Theorem~\ref{sdp}.  
One can check Table~\ref{proof} is equivalent to (\ref{maineq}--\ref{mainan}).  
\end{proof}

We remark that more computations can be carried out for $r>400$. For those large $r$'s,  in order to get a non-trivial bound (less than Gerzon's bound) for $\alpha \leq 1/9$, one can apply the procedure like what we have done for $\alpha=1/7$ in Example \ref{ex:rb7}. According to Table~\ref{known}, Theorem~\ref{thm:main}, and 
 our further experiments, we propose a conjecture below. 

\begin{conjecture}\label{conjecture}
For any $r$, if $m$ is the largest positive integer such that $(2m+1)^2\leq r + 2$, then an upper bound  of maximum number of equiangular lines in ${\mathbb R}^r$ is either 
$\frac{4r\left(m+1\right)\left(m+2\right)}{\left(2m+3\right)^2-r}$ or $\frac{\left(\left(2m+1\right)^2-2\right)\left(\left(2m+1\right)^2-1\right)}{2}$. 
\end{conjecture}

\section{Experiments}\label{experiment}
We compare the relative bounds for $\alpha=1/5$ given by Corollary~\ref{cry:rb5} and the basic SDP method (Theorem \ref{sdp}) without pillar decomposition 
\citep{BaYu14}
in 
Figure~\ref{fig:bound5}.  For each $r$ between $44$ and $400$, we compute the 
two upper bounds by Corollary~\ref{cry:rb5} and Theorem \ref{sdp} for $\alpha=1/5$, respectively. 
In Figure~\ref{fig:bound5}, we mark by blue plusses ``$+$" the computed upper bound according to Corollary~\ref{cry:rb5}, and we mark 
the bounds due to Theorem \ref{sdp} (SDP bounds without pillar decomposition) by red stars ``$\star$". We also draw Gerzon's bound as a black curve.

 Similarly, in the 
Figure~\ref{fig:bound7}, we compare the relative bounds for $\alpha=1/7$ computed by the procedure shown in Example~\ref{ex:rb7} and the SDP method without pillar decomposition.

\begin{figure}[h]
\begin{minipage}{0.5\linewidth}
\caption{\footnotesize{Comparing Corollary \ref{cry:rb5} and Theorem \ref{sdp}}
}
\label{fig:bound5}
\begin{center}
\includegraphics[height=2.6 in, width=2.9 in]{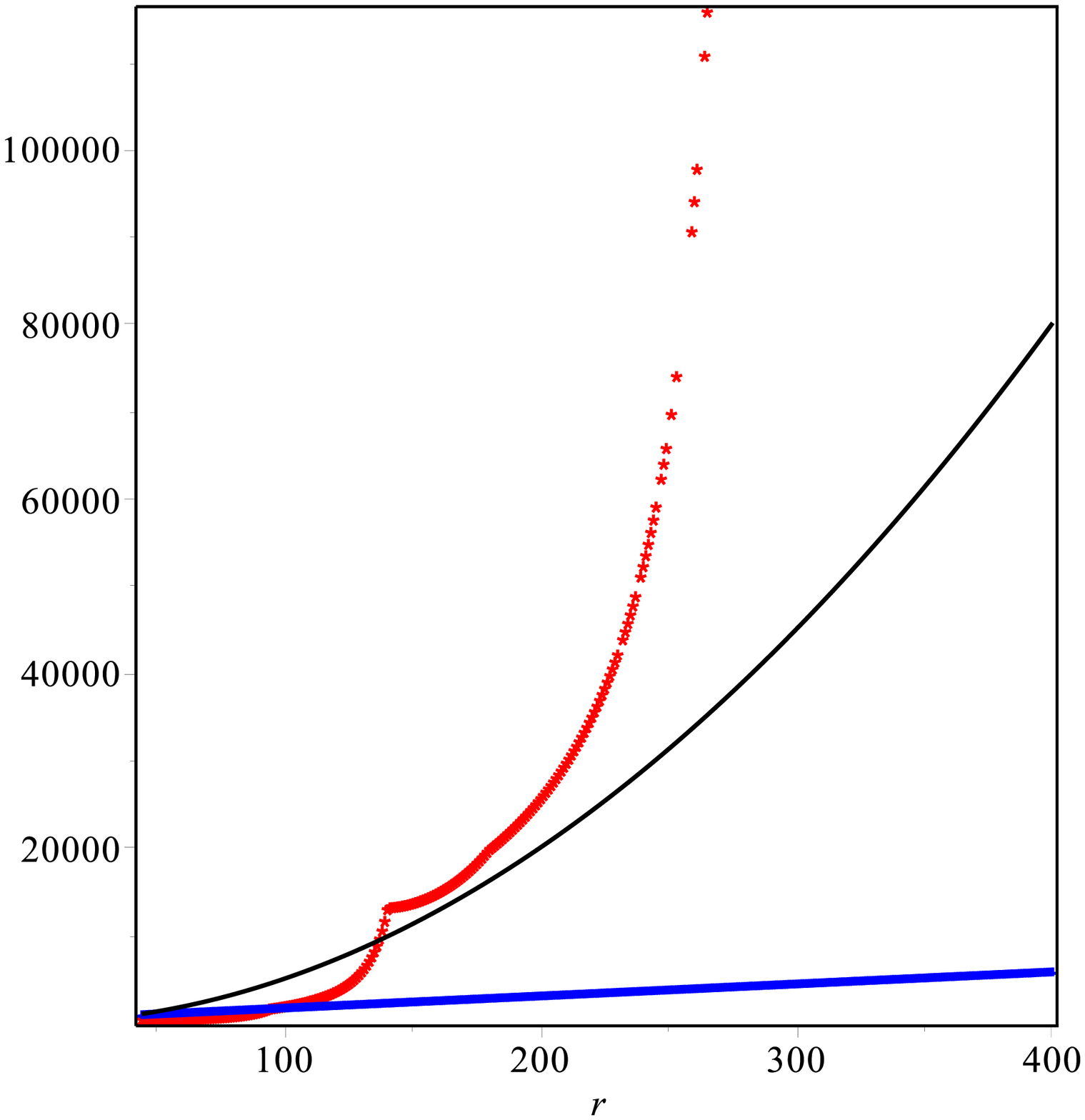}
\end{center}
{\scriptsize
\begin{center}
\textcolor{black}{--- : $\frac{r(r+1)}{2}$}\;\;
\textcolor{red}{$\star\star\star$ : Theorem \ref{sdp} }\;\;
\textcolor{blue}{{\tiny +++} : Corollary \ref{cry:rb5} }
\end{center}
}
\end{minipage}
~~
\begin{minipage}{0.5\linewidth}
\caption{\footnotesize{Comparing Example \ref{ex:rb7} and SDP bound} 
}
\label{fig:bound7}
\begin{center}
\includegraphics[height=2.6 in, width=2.9 in ]{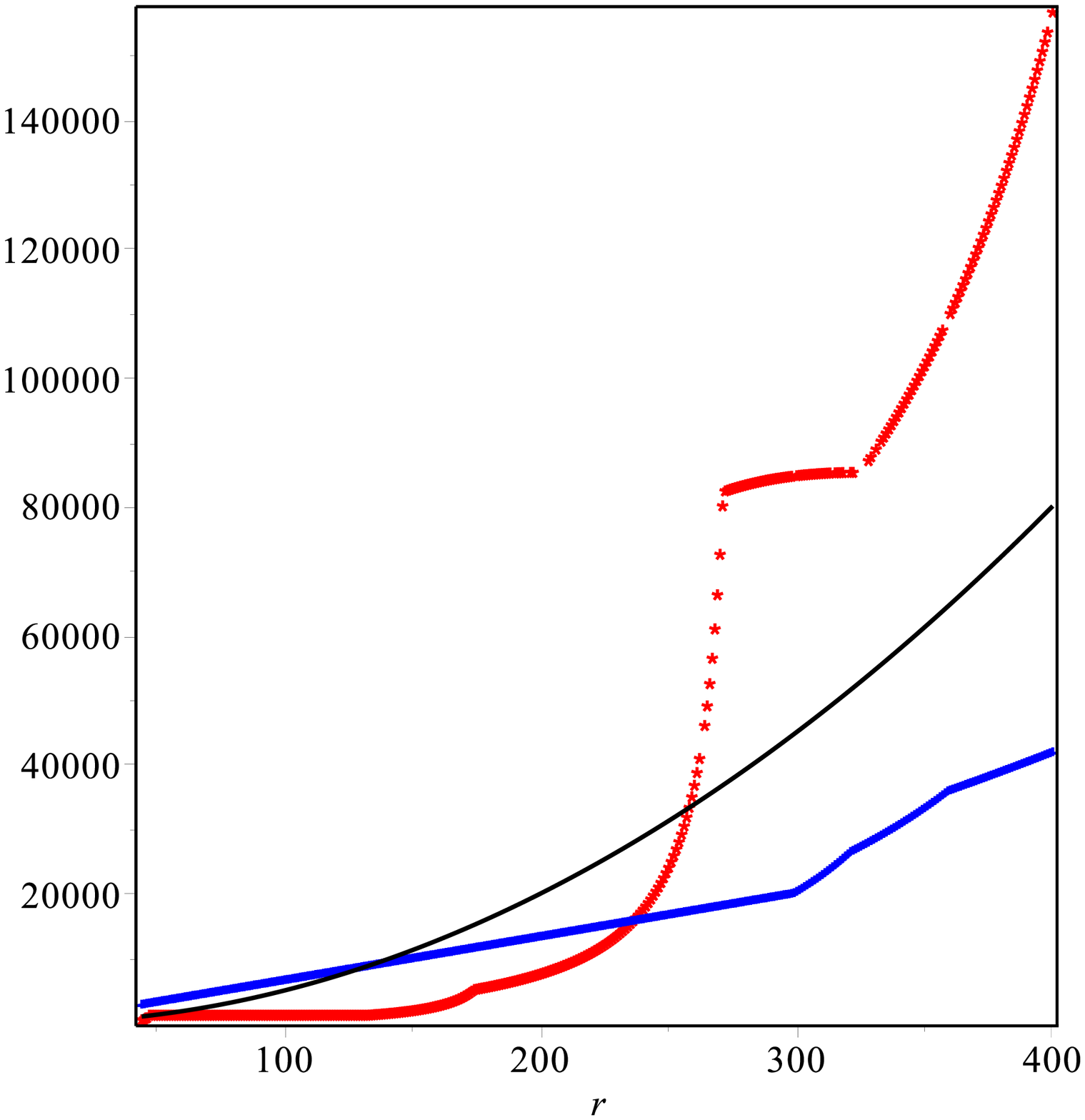}
\end{center}
{\scriptsize
\begin{center}
\textcolor{black}{--- : $\frac{r(r+1)}{2}$}\;\;
\textcolor{red}{$\star\star\star$ : Theorem \ref{sdp} }\;\;
\textcolor{blue}{{\tiny +++} : Example \ref{ex:rb7}}
\end{center}
}
\end{minipage}
\end{figure}

 In our experiments,  we solve SDP by the Matlab software {\bf CVX 3.0 beta} \citep{cvx, gb08}.  There are many  SDP solvers provided by {\bf CVX 3.0 beta} and the computational results presented in this paper are computed by {\bf sdpt3} \citep{sdpt31999, sdpt3}. The computation is carried out by a 3.20GHz Interl(R) Core(TM) i5-4460 processor under x86\_64 GNU/Linux. Our Matlab version is {\bf R2016a}. The code used for the calculations in this paper may be found at \citep[code]{link2018}. The output of {\bf sdpt3} is a floating number. We take the nearest integer of this floating number.

We see in Figures~\ref{fig:bound5}--\ref{fig:bound7} for $44\leq r\leq 400$ that
the ``red'' SDP bound is smaller when $r$ is small, it increases dramatically at some $r$, and it eventually goes beyond the Gerzon's bound.
Further, our ``blue" bound is larger when $r$ is small, but it increases slower, and when $r$ is sufficiently large, 
it always gives non-trivial upper bound 
which is much smaller than either the SDP bound or the Gerzon's bound. Some further comments are as follows. See \citep[table.pdf]{link2018} for the concrete data. 

{\bf (Comment 1).} Concerning Figure~\ref{fig:bound5}, when $44\leq r\leq 93$, the SDP bound without pillar decomposition is smaller than the bound in  Theorem~\ref{thm:rb5}. 
When $94\leq r\leq 400$, the bound in Theorem~\ref{thm:rb5} is smaller (generally much smaller) than the SDP bound. 
One can check the concrete bounds for $r\geq 94$ in \citep[Table.pdf]{link2018} and compare them
with the SDP bounds shown in the red parentheses. The SDP bounds in the red parentheses which we computed are the same as the data shown in 
\citep{BaYu14}, which is to be expected. 
As an example of Theorem~\ref{thm:rb5} outperforming SDP, the bound given by Theorem~\ref{thm:rb5} for $r=137$ is $2015$ while 
the SDP bound is $9528$. 
When $137\leq r\leq 400$, the SDP bound is greater (eventually much greater) than Gerzon's bound, which is consistent with 
the results obtained in \citep{BaYu14}. On the other hand, for any 
$44\leq r\leq 400$, the bound in Theorem~\ref{thm:rb5} is smaller than Gerzon's bound. 
For a range of larger $r$, for instance $r=266$--$400$ and some discrete $r$ such as $231$, $238$ and so on, {\bf sdpt3} failed to compute the SDP bound.  That is the reason why the red markers in Figure~\ref{fig:bound5} are not as continuous as the blue markers. 

{\bf (Comment 2).} Similarly, we note for Figure~\ref{fig:bound7} that when $44\leq r\leq 235$, the SDP bound is smaller.
When $235\leq r\leq 400$, the ``blue'' bound is smaller (generally much smaller) than the SDP bound. 
When $259\leq r\leq 400$, the SDP bound is greater  (eventually much greater) than Gerzon's bound. 
For some large $r$, {\bf sdpt3} failed to compute the SDP bound. 

In conclusion, practically, for $\alpha=1/5$ or $1/7$ and for any $r$, we can compute both upper bounds for $s_\alpha(r)$: the SDP bound and the ``blue'' bound, and then pick up the smaller one.  

An extra remark is that for $\alpha < 1/7$, experiments show that the SDP bounds are always no greater than 
the Gerzon's bound for $44\leq r\leq 400$. 
See the SDP bounds in the columns 
``$1/9$", $\ldots$, ``$1/27$"
in 
\citep[Table.pdf]{link2018}. 
A possible reason is that 
$400$ is not large enough for the SDP bound $(r\leq 400, \alpha<1/7)$ to go beyond the 
Gerzon's bound.

\section{Acknowledgments}
We are particularly indebted to Bernd Sturmfels for introducing the authors to each other. We thank Toh Kim Chuan for his nice answers to our questions about {\bf sdpt3}. We are very appreciative of Wei-Hsuan Yu and Gary Greaves, who pointed out the newest bounds on $s(17)$--$s(20)$. Finally, we thank Eva Maria Feichtner for her helpful comments on this work.  The authors were supported in part by the Zentrum f\"ur Forschungsf\"orderung der Uni Bremen Explorationsprojekt ``Hilbert Space Frames and Algebraic Geometry.''

\section*{References}

{\footnotesize
\bibliographystyle{elsart-harv}
\bibliography{etf}
}
\begin{appendices}
\section{Proof of Theorem~\ref{general} }\label{App:AA}

In Appendix~\ref{App:AA}, let $X\subset {\mathbb R}^r$ be an equiangular set with angle $\alpha$, where $1/\alpha$ is an odd number greater or equal to $3$. Let $K=\kx$. For any $K$-base $\{p_1, \ldots, p_K\}\subset X$ such that 
$\langle p_i, p_j\rangle=- \alpha$ $(\forall~1\leq i< j\leq K)$, 
let $P$ be the linear subspace of ${\mathbb R}^r$ spanned by $p_1, \ldots, p_K$ and let ${P}^{\bot}$ be the orthogonal complement of $P$ in ${\mathbb R}^r$. 
For any  $x\in X\backslash \{p_1, \ldots, p_K\}$, 
the following lemma originally comes from \cite[Section 4]{LS1973}. 
The arguments to prove the following lemma are so similar for both cases $K = (1/\alpha)+1$ and $K < (1/\alpha)+1$ that we combine them, even though $K < (1/\alpha) +1$ applies to  Appendix~\ref{App: AB}. 
  \begin{lemma}\label{lm:bound1973} 
Let $X$ be an equiangular set in $\bR^r$ with $\alpha$.  We set $K = K_\alpha(X)$.  Then for each equivalence class $\ex \subset X$ {w.r.t} any fixed $K$-base $\{p_1, \ldots, p_K\}$, 
\begin{equation*}
|\ex|~\leq~  \left\{\begin{array}{ll} r- K + 1   +  \lfloor 2\alpha\frac{r-K+1 }{1-\alpha} \rfloor, & K = (1/\alpha)+1\vspace*{1mm} \\
 s(r,\beta,\gamma), & K < (1/\alpha) +1, \enskip x \in X(K,n), \enskip n=1, \ldots, \lfloor K/2\rfloor\end{array} \right.
\end{equation*}
where $X(K, n)$ is defined as in (\ref{eq:x}), 
$\beta=(\alpha-\ell(K,n))/(1-\ell(K,n))$ and $\gamma=(-\alpha-\ell(K,n))/(1-\ell(K,n))$. 
\end{lemma}

\begin{proof}
Let $\ex=\{x_1, \ldots, x_s\}$. By Corollary \ref{cry:ec}, 
we can assume that for each $i=1, \ldots, s$, 
\[x_i=h+c_i, \;\;\;\text{where}\;\;h\in P \;
\text{and}\;c_i\in P^{\bot}.\] 
Then
\[ G(\ex)~=~\ip{h}{h} J + G(c_1, \ldots, c_s).\] 
Note the norm squared  of $c_i$ is $\langle c_i, c_i\rangle=\langle x_i, x_i\rangle -  \langle h, h\rangle = 1- \langle h, h\rangle$.
By Proposition \ref{lm:project}, if $K=(1/\alpha)+1$, then $\langle h, h\rangle=\alpha < 1$ (note $0<\alpha<1$ by Definition \ref{def:def1}), and if 
 $K<(1/\alpha)+1$, then $\langle h, h\rangle<1$ (see Formula (\ref{eq:beta})). 
So we always have $1- \langle h, h\rangle\neq 0$.  If we normalize the vectors $c_1, \ldots, c_s$ as $c'_1, \ldots, c'_s$, then the Gramian matrix of  $c'_1, \ldots, c'_s$ is
\[G(c'_1, \ldots, c'_s) ~=~\frac{G(c_1, \ldots, c_s)}{1-\ip{h}{h}}~=~\frac{G(\ex) - \ip{h}{h}J}{1-\ip{h}{h}}.\]
Obviously, this matrix has $1$'s along the diagonal. 
Note $G(\ex)$ has $\alpha$ or $-\alpha$ as the off-diagonal entries. So
$G(c'_1, \ldots, c'_s) $ has two different numbers $(\alpha-\ip{h}{h})/(1-\ip{h}{h})$ or $(-\alpha-\ip{h}{h})/(1-\ip{h}{h})$ as the off-diagonal entries.
 So $\{c'_1, \ldots, c'_s\}$ is
a spherical two-distance set  such that 
\begin{equation*}\label{sph}
\langle c'_i, c'_j \rangle~ =~
\begin{cases}
1,& i = j\\
\frac{\alpha-\ip{h}{h}}{1-\ip{h}{h}},~  \text{or}~\frac{-\alpha-\ip{h}{h}}{1-\ip{h}{h}}, & i \neq j.
\end{cases}
\end{equation*}
Thus $\overline{x} \leq s(r,\beta,\gamma)$, with $\beta =(\alpha-\ip{h}{h})/(1-\ip{h}{h})$ and $\gamma = (-\alpha-\ip{h}{h})/(1-\ip{h}{h})$.
The bounds follow from the values of $\ip{h}{h}$ in Proposition~\ref{lm:project} 
\[
\ip{h}{h}~=~
\begin{cases}
\alpha, & K = (1/\alpha)+1\\
\ell(K,n), & K < (1/\alpha)+1,~ x \in X(K,n), ~n=1, \ldots, \lfloor K/2\rfloor
\end{cases}
\]
and, for $K = (1/\alpha)+1$, since $\beta = 0$ and $\gamma=-2\alpha/(1-\alpha)$, it follows from \cite[Theorem 4.1]{LS1973} that 
$s(r,0, -2\alpha/(1-\alpha))\leq  r- K   +1 +  \lfloor 2\alpha(r-K + 1)/(1-\alpha) \rfloor$.
\end{proof}

\begin{lemma}\label{lm:project1973}
For  any $x\in X\backslash \{p_1, \ldots, p_K\}$, suppose 
\[(\langle x, p_1\rangle, ~\ldots, ~\langle x, p_K\rangle)^{\top} = \alpha\cdot \left(\epsilon^{(1)}, ~\ldots,~ \epsilon^{(K)}\right)^{\top},\;\;\;\text{where}\;\;\;\epsilon^{(i)}~=~\pm 1.\]
If $K=(1/\alpha)+1$, 
then $\epsilon^{(1)}+\cdots+\epsilon^{(K)}=0$.
\end{lemma}
\begin{proof}
By Proposition~\ref{lm:easy}, $p_1+\cdots+p_K=0$ since $K=(1/\alpha) + 1$. 
 Suppose we decompose $x$ in $P$ and $P^{\bot}$  as
\[x~=~h+c, \;\;\;\text{where}\;\;h\in P \;
\text{and}\;c\in P^{\bot}.\]
Then $\langle h, p_1+\cdots+p_K\rangle=0$. On the other hand, we have
\[\langle h, p_1+\cdots+p_K\rangle~=~\langle h, p_1\rangle + \cdots + \langle h, p_K\rangle~=~\langle x, p_1\rangle + \cdots + \langle x, p_K\rangle~=~\alpha\left(\epsilon^{(1)}+\ldots+\epsilon^{(K)}\right).\]
Since $\alpha\neq 0$,  
$\epsilon^{(1)}+\ldots+\epsilon^{(K)}~=~0$.
\end{proof}

\noindent 
{\bf Proof of Theorem~\ref{general} .}
Note $K=(1/\alpha)+1$ is even since $1/\alpha$ is odd.  
By Lemma \ref{lm:project1973} and Lemma~\ref{lm:ec}, there are 
$\frac{1}{2}\binom{K}{ K/2}$ equivalence classes $\ex$ in $X$. More specifically, 
we index
all the $(1, -1)$-vectors in ${\mathbb R}^K$ with exactly $ K/2$ $1$-coordinates and the first coordinate $1$ as
\[\epsilon_{m},  \;\;m = 1, \ldots, \frac{1}{2}\binom{K}{K/2}\]
and for each $\epsilon_m$, we define
\[\ex_{m}~=~\{x\in X\backslash \{p_1, \ldots, p_K\}|(\langle x, p_1\rangle, ~\ldots, ~\langle x, p_K\rangle) ~= ~\pm \alpha\cdot \epsilon_{m}\}.\]
Then we have a partition of $X$
 \begin{equation}\label{eq:ec1973}
 X~ = ~\{p_1, \ldots, p_K\}\bigcup\bigcup\limits_{m=1}^{ \frac{1}{2}\binom{K}{K/2}}  \ex_{m}.
 \end{equation}
By the partition (\ref{eq:ec1973}) and Lemma \ref{lm:bound1973}, we have 
\[|X|~\leq~ K + \frac{1}{2}\binom{K}{K/2}\left(r -  K + 1 +  \lfloor 2\alpha\frac{r-K+1}{1-\alpha} \rfloor \right). \;\;\;\;\;\;\;\;\;\;\; \Box\]

\section{Proof of Theorem~\ref{th:s2}}\label{App: AB}
In Appendix~\ref{App: AB}, let $X\subset {\mathbb R}^r$ be an equiangular set with angle $\alpha$ and let $K$ be  $K=\kx < (1/\alpha)+1$. Again, we fix a $K$-base $\{p_1, \ldots, p_K\}\subset X$ such that 
$\langle p_i, p_j\rangle=- \alpha$ $(\forall~1\leq i< j\leq K)$.
Let $P$ be the linear subspace of ${\mathbb R}^r$ spanned by $p_1, \ldots, p_K$ and let ${P}^{\bot}$ be the orthogonal complement of $P$ in ${\mathbb R}^r$. 

We remark that the bound for $|\ex|$ given by Lemma~\ref{lm:bound1973}  is not in closed form. We could run an SDP tool to compute an upper bound for $s\left(r, \beta, \gamma \right)$ or possibly apply another result about two-distance spherical sets. However, if $1\leq n\leq K- ((1/\alpha)+1)/2$, we can modify  Lemma~\ref{lm:bound1973}  and
 derive explicit bounds, see Lemmas \ref{lmbound2}--\ref{lmbound1} below.


\begin{lemma}\label{lmbound2}
If $1<n< K- ((1/\alpha)+1)/2$, then for any $x\in X(K,n)$, $|\ex|~\leq~  r+1$.
\end{lemma}
\begin{proof}
By Proposition~\ref{lm:project} (\ref{eq:beta}), if $n< K- ((1/\alpha)+1)/2$, then $\alpha<\ell(K,n)<1$. So 
both $\beta= (\alpha-\ell(K,n))/(1-\ell(K,n))$ and $\gamma = (-\alpha-\ell(K,n))/(1-\ell(K,n))$ are negative. 
By Lemma~\ref{lm:bound1973}   and \cite[Theorem 3.2, last case]{BaYu13}, 
$|\ex|~\leq~ s\left(r, \beta, \gamma \right) ~\leq~ r+1$.
\end{proof}

\begin{lemma}\label{lmbound3}
If $n= K-  ((1/\alpha)+1)/2$, then for any $x\in X(K,n)$,
$|\ex|~\leq~   r- K   +  \lfloor 2\alpha\frac{r-K }{1-\alpha} \rfloor$.
\end{lemma}

\begin{proof}
By Proposition~\ref{lm:project} (\ref{eq:beta}), if $n= K- \ ((1/\alpha)+1)/2$, then $\ell(K,n)=\alpha$. So $\beta= (\alpha-\ell(K,n))/(1-\ell(K,n))=0$ and
$\gamma= (-\alpha-\ell(K,n))/(1-\ell(K,n))=-2\alpha/(1-\alpha)$.  By \cite[Theorem 4.1]{LS1973}, 
$|\ex|\leq  r- K   +  \lfloor 2\alpha (r-K)/(1-\alpha) \rfloor$. 
\end{proof}

\begin{lemma}\label{lmbound1}
For any $x\in X(K,1)$, 
\begin{equation*}\label{bound1}
|\ex|~\leq~ 
\begin{cases}
r-K,  & 1\geq K - \frac{(1/\alpha)+1}{2} 
\\
\frac{1-\alpha}{\ell(K,1)-\alpha},
& 1<K- \frac{(1/\alpha)+1}{2},
\end{cases}
\;\;\;\;\;\text{where}\;\;0<\alpha<1.
\end{equation*}
\end{lemma}
\begin{proof}
Let $\ex=\{x_1, \ldots, x_s\}$. 
By Corollary \ref{cry:ec}, we can assume that for each $i=1, \ldots, s$, 
\[x_i=h+c_i, \;\;\;\text{where}\;\;h\in P \;
\text{and}\;c_i\in P^{\bot},\]
and for all $x \in \ex\subset X(K, 1)$, there exist $K-1$ vectors among $p_1, \ldots, p_K$, say $p_1, \ldots, p_{K-1}$, such that for any $i=1, \ldots, K-1$, $\ip{p_i}{x} = -\alpha$.

Assume there exist $x, \tilde{x}\in\ex\subset X(K,1)$ such that $\langle  x, \tilde{x}\rangle=-\alpha$. 
 Then the $K+1$ vectors $p_1, \ldots, p_{K-1}, x, \tilde{x}\in X$ have pairwise inner product  $-\alpha$ and hence 
$\kx\geq K+1$, which contradicts to the hypothesis that $\kx=K$. So we must have $\langle  x, \tilde{x}\rangle=\alpha$ for any $x, \tilde{x}\in\ex$.

By the above argument, we have $G(\ex)=(1-\alpha)I + \alpha J$.  By Proposition \ref{lm:project}, $\langle  h, h\rangle=\ell(K, 1)$.
Thus,  
\begin{equation}\label{eqbound1}
G(c_1, \ldots, c_s)~=~G(\ex) - \ell(K,1) \cdot J~=~  \left(1 - \alpha \right)I - \left(\ell(K,1)-\alpha\right)J .
\end{equation}
By Lemma \ref{pp:basic}, the above matrix has two eigenvalues 
\begin{equation}\label{eq:lmbound1}
\lambda_1~=~1-\ell(K,1) + \left(s-1\right)\left(\alpha-\ell(K,1)\right),\;\;\lambda_2~=~1-\alpha.
\end{equation}
For the first case $1\geq K- ((1/\alpha)+1)/2$, 
by Proposition \ref{lm:project} (\ref{eq:beta}),  we have $\ell(K,1)\leq \alpha<1$ in this case. Hence,  it is seen from (\ref{eq:lmbound1}) that  
both eigenvalues $\lambda_1$ and $\lambda_2$ are strictly positive. 
So the matrix $G(\ex) -  \ell(K,1) \cdot J$ is full rank and the rank is $s$. 
Note  $c_1, \ldots, c_s\in P^{\bot}$ and $P^{\bot}$ has dimension $r-K$.   Thus, by Lemma~\ref{lem:gramspan}, we have $s\leq r-K$.

Now we discuss the second case $1<K-  ((1/\alpha)+1)/2$. 
Note $G(c_1, \ldots, c_s)$ is positive semidefinite and hence all its eigenvalues should be non-negative. By checking $\lambda_1$,  we have 
\[\lambda_1~=~1-\ell(K,1) + \left(s-1\right)\left(\alpha-\ell(K,1)\right)\geq 0 \Longleftrightarrow s\leq  \frac{1-\alpha}{\ell(K,1)-\alpha}.\] 
(We remark that by Proposition \ref{lm:project} (\ref{eq:beta}), 
if $1<K- ((1/\alpha)+1)/2$,  then $ \alpha<\ell(K,1)<1$. So the above bound for $s$ is non-trivial.)
\end{proof}

\section{\bf Proof of Lemma \ref{lm:rb4}}\label{app:AC}
The goal of this subsection is to prove Lemma \ref{lm:rb4}. Suppose we have an equiangular set $X$ in ${\mathbb R}^r$ with angle $1/5$ and 
we have $K=\km(X)=4$. Fix a $K$-base $\{p_1, p_2, p_3, p_4\}$. 
Define $X(K, n)$ as in (\ref{eq:x}). 
Then the partition (\ref{eq:decompose}) becomes 
\begin{align}\label{eq:dec4}
X
~=~\{p_1, p_2, p_3, p_4\} \bigcup X(4,1) \bigcup X(4,2).
\end{align}
By Theorem \ref{thm:ec}, there are $\binom{4}{1}=4$ equivalence classes in $X(4,1)$, say $\ex_m$ $(m=1, 2, 3, 4)$, with
\[|X(4,1)| = \sum_{m=1}^4|\ex_m|.\]
and $\frac{1}{2}\binom{4}{2}=3$ equivalence classes in $X(4,2)$, say $\ex_m$ $(m=5, 6, 7)$ with 
\[|X(4,2)| = \sum_{m=5}^7|\ex_m|.\]
So by the partition in (\ref{eq:dec4}), 
we have
\begin{align}\label{eq:newdec4}
|X|~=~4 + |X(4,1)| + |X(4,2)|~=~4+\sum_{m=1}^4|\ex_m| + \sum_{m=5}^7|\ex_m|.
\end{align}
Below, we discuss the upper bound for each $|\ex_m|$ in Lemmas~\ref{lm:rb41} and~\ref{lm:rb42}. After that, we give the proof of 
Lemma \ref{lm:rb4}. 

\begin{lemma}\label{lm:forrb41}
If a real $(s+1)\times (s+1)$-matrix ($s>2$)
\[G ~= ~\left(\begin{array}{cccc}
a &  \hdots & 0 & i_1\\  
\vdots & \ddots& \vdots & \vdots \\ 
0 & \hdots & a   & i_s\\
i_1&  \hdots &i_s & b
\end{array}\right) 
\]
is positive semidefinite, then  $\sum_{k=1}^s i_k^2\leq ab$. 
\end{lemma}
\begin{proof}
It is straightforward to compute that the characteristic polynomial of the matrix $G$ is
\[\left(\lambda-a\right)^{s-2}\left(\lambda^2 - (a+b)\lambda + ab - \sum_{k=1}^si_k^2\right).\] 
Because $G$ is positive semidefinite, we have all non-negative eigenvalues. Hence, 
the product of last two eigenvalues should be non-negative. That means 
$ab - \sum_{k=1}^si_k^2 ~\geq~0$.
\end{proof}
\begin{lemma}\label{lm:rb41}
If $x\in X(4,1)$ and there exists another vector $y \in X$ such that $y \not\in \ex$, 
then we have the statements below. 
\begin{itemize}
\item[{\em (1)}]If $y\in X(4,1)$, then $|\ex|\leq 36$.
\item[{\em (2)}]If $y\in X(4,2)$, then $|\ex|\leq 39$.
\end{itemize}
\end{lemma}
\begin{proof}
Since $x\in X(4,1)$, by the definition of $X(4,1)$ in (\ref{eq:x}), 
 there is exactly $1$ positive inner product or $4-1=3$ positive inner products among $\langle x, p_i\rangle,\; i=1, 2, 3, 4$.
Without loss of generality, we assume 
\[ (\;\langle x, p_1\rangle,  \;\;\langle x, p_2\rangle, \;\;\langle x, p_3\rangle, \;\; \langle x, p_4\rangle\;)^{\top}\;=\;\frac{1}{5}(-1, 1, 1, 1)^{\top}.\]
If it is in the other cases, one can derive the same conclusion because of symmetry. 

Suppose $x = h + c$ where $h \in P$ and $c\in P^{\bot}$.
By Proposition \ref{lm:project},  
\begin{align}\label{eq:lmrb41hp}
h ~=~ a^{(1)}p_1+\cdots+a^{(4)}p_4~=~ \frac{p_2 + p_3 + p_4}{3}, 
\end{align}
where\[\left(a^{(1)}, a^{(2)}, a^{(3)}, a^{(4)}\right)~=~\alpha~\left((1+\alpha)I_4-\alpha J_4\right)^{-1}~(-1, 1, 1, 1)^{\top}~=~\left(0, 1/3, 1/3, 1/3\right).\]
Here, note $\left((1+\alpha)I_4-\alpha J_4\right)^{-1}\big\vert_{\alpha=1/5} = (5/4)I_4+(5/12)J_4$, and  by (\ref{eq:l}), 
\begin{align}\label{eq:lmrb41h}
\langle h, h\rangle~ =~ \ell(4,1)~ = ~\frac{\alpha^2 \left(4\alpha n \left(n-K\right)+K\left(1+\alpha\right)\right)}{(1+\alpha)\left(1-\left(K-1\right)\alpha\right)}\bigg\vert_{\alpha=1/5,\;\; n=1, \;\;K=4}~ = ~\frac{1}{5}. 
\end{align}
Assume 
\[ (\;\langle y, p_1\rangle,  \;\;\langle y, p_2\rangle, \;\;\langle y, p_3\rangle, \;\; \langle y, p_4\rangle\;)^{\top}\;=\; \frac{1}{5}\epsilon.\]
Since $y\not\in \ex$, by Lemma \ref{lm:ec}, $\epsilon$ can not be $\pm(-1, 1, 1, 1)^{\top}$.
\begin{itemize}
\item[(1)]
If $y\in X(4,1)$, then $\epsilon$ can be one of the six vectors below
\[ \pm(1, -1, 1, 1)^{\top}, \;\; \pm(1, 1, -1, 1)^{\top}, \;\; \pm(1, 1, 1, -1)^{\top}.\]
Suppose $y = g + d$ where $g \in P$ and $d\in P^{\bot}$.
According to those above six possible $\epsilon$'s, by Proposition \ref{lm:project},  $g$ will be respectively 
\begin{align}\label{eq:lmrb41gp}
\pm\frac{p_1 + p_3 + p_4}{3}, \;\; \pm\frac{p_1 + p_2 + p_4}{3},\;\;\pm\frac{p_1 + p_2 + p_3}{3}.
\end{align}
However, remark that in all cases, we have
\begin{align}\label{eq:lmrb41g}
\langle g, g\rangle = 1/5.
\end{align}
Let $\ex=\{x_1, \ldots, x_s\}$. 
By Corollary \ref{cry:ec}, we can assume that for each $i=1, \ldots, s$, 
\[x_i=h+c_i, \;\;\;\text{where}\;\;c_i\in P^{\bot}.\]
Recall the proof of Lemma~\ref{lmbound1} that for any different $x_i, x_j\in \ex \subset X(4, 1)$, we have 
\begin{align}\label{eq:lmrb41x}
\ip{x_i}{x_j}=1/5.
\end{align}  
Thus, by (\ref{eq:lmrb41h}), (\ref{eq:lmrb41g}) and (\ref{eq:lmrb41x}), 
the Gramian matrix of $c_1, \ldots, c_s, d$ is
{\footnotesize
\begin{align*}
\left(\begin{array}{cccc}
\ip{c_1}{c_1} &  \hdots & \ip{c_1}{c_s} & \ip{c_1}{d}\\  
\vdots & \ddots& \vdots & \vdots \\ 
\ip{c_s}{c_1}  & \hdots & \ip{c_s}{c_s}   & \ip{c_s}{d}\\
\ip{d}{c_1} &  \hdots & \ip{d}{c_s} & \ip{d}{d}
\end{array}\right)
&~=~
\left(\begin{array}{cccc}
\ip{x_1}{x_1} &  \hdots & \ip{x_1}{x_s} & \ip{x_1}{y}\\  
\vdots & \ddots& \vdots & \vdots \\ 
\ip{x_s}{x_1}  & \hdots & \ip{x_s}{x_s}   & \ip{x_s}{y}\\
\ip{y}{x_1} &  \hdots & \ip{y}{x_s} & \ip{y}{y}
\end{array}\right)-
\left(\begin{array}{cccc}
\ip{h}{h} &  \hdots & \ip{h}{h} & \ip{h}{g}\\  
\vdots & \ddots& \vdots & \vdots \\ 
\ip{h}{h}  & \hdots & \ip{h}{h}   & \ip{h}{g}\\
\ip{g}{h} &  \hdots & \ip{g}{h} & \ip{g}{g}
\end{array}\right)\\
&~=~
\left(\begin{array}{cccc}
1 &  \hdots & 1/5 & \ip{x_1}{y}\\  
\vdots & \ddots& \vdots & \vdots \\ 
1/5 & \hdots & 1   & \ip{x_s}{y}\\
\ip{y}{x_1} &  \hdots & \ip{y}{x_s} & 1
\end{array}\right)-
\left(\begin{array}{cccc}
1/5&  \hdots & 1/5&  \ip{h}{g}
\\  
\vdots & \ddots& \vdots & \vdots \\ 
1/5  & \hdots & 1/5   &  \ip{h}{g}\\
 \ip{h}{g}
 &  \hdots &  \ip{h}{g}
 & 1/5
\end{array}\right)\\
&~=~
\left(\begin{array}{cccc}
4/5 &  \hdots & 0 & \ip{x_1}{y}- \ip{h}{g}\\  
\vdots & \ddots& \vdots & \vdots \\ 
0 & \hdots & 4/5   & \ip{x_s}{y}- \ip{h}{g}\\
\ip{y}{x_1}- \ip{h}{g} &  \hdots & \ip{y}{x_s}- \ip{h}{g} & 4/5
\end{array}\right).
\end{align*}
}
For $i=1, \ldots, s$, let $i_k = \ip{x_k}{y} - \ip{h}{g}~~(=  \ip{y}{x_k} - \ip{h}{g})$. 
Here in order to apply Lemma \ref{lm:forrb41}, we assume $s>2$; otherwise, the conclusion $s\leq 36$ we want to prove  will be naturally true.
Then by Lemma \ref{lm:forrb41}, $\sum_{k=1}^si_k^2 \leq \left(\frac{4}{5}\right)^2$. 
Note also $ \ip{x_k}{y} $ is $\pm 1/5$ since $\ex\cup \{y\}\subset X$ is equiangular. By (\ref{eq:lmrb41hp}), (\ref{eq:lmrb41gp}) and the fact that $\ip{p_i}{p_j}=-1/5$ for any $i\neq j$, one can compute directly that $\ip{h}{g}$ is $\pm 1/15$. 

So 
\[|i_k|~=~|\ip{x_k}{y} - \ip{h}{g}|~\geq~  \left(1/5\right)- \left(1/15\right) ~=~ \left(2/15\right).\] Thus,
\[s\cdot\left(\frac{2}{15}\right)^2\leq \sum_{k=1}^si_k^2 \leq \left(\frac{4}{5}\right)^2 \quad \Rightarrow \quad s\leq \frac{\left(4/5\right)^2}{\left(2/15\right)^2} =36.\]
\item[(2)]
If $y\in X(4,2)$, then $\epsilon$ has the six possibilities below  
\[ \pm(1, 1, -1, -1)^{\top}, \;\; \pm(1, -1, 1, -1)^{\top}, \;\; \pm(1, -1, -1, 1)^{\top}.\]
According to these six $\epsilon$'s, by Proposition \ref{lm:project},  $g$ will respectively be
\[\pm\frac{p_1 + p_2 - p_3 - p_4}{6}, \;\; \pm\frac{p_1 - p_2 + p_3 - p_4}{6},\;\;\pm\frac{p_1 - p_2 - p_3 + p_4}{6}.\]
However, in all cases, we have $\langle g, g\rangle = 2/15$ and $\langle h, g\rangle=\pm 1/15$.
Let $\ex=\{x_1, \ldots, x_s\}$. 
By Corollary \ref{cry:ec}, we can assume that
\[x_i=h+c_i\;\;\;\text{for each}~i=1, \ldots, s,\;\;\text{where}\;\;c_i\in P^{\bot}. \]Similar to what we have done in the case (1), one can compute that the Gramian matrix of $c_1, \ldots, c_s, d$ is
{\footnotesize
\begin{align*}
\left(\begin{array}{cccc}
\ip{c_1}{c_1} &  \hdots & \ip{c_1}{c_s} & \ip{c_1}{d}\\  
\vdots & \ddots& \vdots & \vdots \\ 
\ip{c_s}{c_1}  & \hdots & \ip{c_s}{c_s}   & \ip{c_s}{d}\\
\ip{d}{c_1} &  \hdots & \ip{d}{c_s} & \ip{d}{d}
\end{array}\right)
~=~
\left(\begin{array}{cccc}
4/5 &  \hdots & 0 & \ip{x_1}{y}- \ip{h}{g}\\  
\vdots & \ddots& \vdots & \vdots \\ 
0 & \hdots & 4/5   & \ip{x_s}{y}- \ip{h}{g}\\
\ip{y}{x_1}- \ip{h}{g} &  \hdots & \ip{y}{x_s}- \ip{h}{g} & 13/15
\end{array}\right).
\end{align*}
}Again, for $i=1, \ldots, s$, let $i_k = \ip{x_k}{y} - \ip{h}{g}$.
By Lemma \ref{lm:forrb41}, $\sum_{k=1}^si_k^2 \leq \frac{4}{5}\cdot\frac{13}{15}  =\frac{52}{75} $. 
Note again $ \ip{x_k}{y} $ is $\pm 1/5$ and $\ip{h}{g}$ is $\pm 1/15$. So 
\[|i_k|~=~|\ip{x_k}{y} - \ip{h}{g}|~\geq~ \left(1/5\right)- \left(1/15\right)~ =~ \left(2/15\right).\] Thus, we have
\[s\cdot\left(\frac{2}{15}\right)^2\leq \sum_{k=1}^si_k^2 \leq \frac{52}{75} \quad \Rightarrow \quad s\leq \frac{52/75}{\left(2/15\right)^2} =39. \]
\end{itemize}

\end{proof}

\begin{lemma}\label{lm:rb42}
If $x\in X(4,2)$, then \[|\ex|~\leq~  s\left(r, 1/13, -5/13 \right).\]
\end{lemma}
\begin{proof} 
We note that $K-(1 + (1/\alpha))/2 = 4 - (1+5)/2 = 1<n=2$.  So we apply Theorem~\ref{th:s2}~(4) to conclude that $|\ex|\leq 
 s\left(r, \beta, \gamma \right)$,
 with $\ell(4,2) = 2/15$, $\beta = \left(1/5 - 2/15\right)/\left(1 - 2/15\right) = 1/13$, and $\gamma = \left(-1/5 - 2/15\right)/\left(1 - 2/15\right) = -5/13$. 
\end{proof}

\noindent
{\bf Proof of Lemma \ref{lm:rb4}.}
Based on Formula (\ref{eq:newdec4}) and Lemma \ref{lm:rb41}, we have $4$ cases to consider.  
\begin{itemize}
\item[]{\bf (Case 1).} If $X(4,2)=\emptyset$ and there is only one non-empty  equivalence class in $X(4,1)$, say $\ex_1$, then by Theorem \ref{th:s2} (1), $|\ex_1|\leq r-4$. So by Formula (\ref{eq:newdec4}), \[|X| ~=~ 4 + |\ex_1| ~\leq~ 4 + r-4~ =~r.\] 
\item[]{\bf (Case 2).} If $X(4,2)=\emptyset$ and there are at least two non-empty equivalence classes in $X(4,1)$, then by Lemma
\ref{lm:rb41} (1), for each equivalence class $\ex_m$, we have $|\ex_m|\leq 36$. So by the formula (\ref{eq:newdec4}), 
\[|X| ~=~ 4 + |X(4,1)| ~=~ 4 + \sum_{m=1}^4 |\ex_m| ~\leq~ 4 + 4\times 36 ~=~148.\]
\item[]{\bf (Case 3).} If $X(4,2)\neq \emptyset$ and
 if there is only one non-empty  equivalence class in $X(4,1)$, say $\ex_1$, then by Lemma \ref{lm:rb41} (2), $|\ex_1|\leq 39$,
 and by Lemma \ref{lm:rb42}, for each $\ex_m\subset X(4,2)$, we have 
 \[|\ex_m|~\leq~ s\left(r, 1/13, -5/13 \right), \;\;\;m=5, 6, 7.\]
  So 
 by Formula (\ref{eq:newdec4}), 
 \[|X| ~= ~4 + |X(4,1)| + |X(4,2)|~ = ~4 + |\ex_1| + \sum_{m=5}^7 |\ex_m|  ~ \leq~  
  43 + 3\cdot s\left(r, 1/13, -5/13 \right).
 \]
\item[]{\bf (Case 4).} If $X(4,2)\neq \emptyset$ and if there are at least two non-empty equivalence classes in $X(4,2)$, then by Lemma
\ref{lm:rb41} (1), for each equivalence class $\ex_m$, we have $|\ex_m|\leq 36$,  and by Lemma \ref{lm:rb42}, for each $\ex_m\subset X(4,2)$, we have 
 \[|\ex_m|~\leq~ s\left(r, 1/13, -5/13 \right), \;\;\;m=5, 6, 7.\]
So by Formula (\ref{eq:newdec4}),
\begin{align*}
|X| &~=~ 4 + |X(4,1)| + |X(4,2)| ~=~ 4 + \sum_{m=1}^7 |\ex_m|\\  
&~\leq~ 4 + 4\times 36 + 3\cdot s\left(r, 1/13, -5/13 \right)~=~
148 + 3\cdot s\left(r, 1/13, -5/13 \right).
\end{align*}
\end{itemize}
Note also for any $r$, by Corollary~\ref{lm:rb4lower}, $s\left(r, 1/13, -5/13 \right)\geq r$.
So among the above $4$ cases, the maximum upper bound of $|X|$ is that in the last case $148 + 3\cdot s\left(r, 1/13, -5/13 \right)$. 
$\Box$

\section{\bf Proof of Lemma \ref{lm:rb5}}\label{app:AD}
The goal of this subsection is to prove Lemma \ref{lm:rb5}. Suppose we have an equiangular set $X$ in ${\mathbb R}^r$ with angle $1/5$ and  $K=\km(X)=5$.  Fix a $K$-base $\{p_1, p_2, p_3, p_4, p_5\}$. 
Define $X(K, n)$ as in (\ref{eq:x}). Then the partition in (\ref{eq:decompose}) becomes 
\begin{align}\label{eq:dec5}
X
~=~\{p_1, \ldots, p_5\} \bigcup X(5,1) \bigcup X(5,2).
\end{align}
Below, we discuss the upper bounds for $|X(5,1)|$ and $|\{p_1, \ldots, p_5\} \cup X(5,2)|$ in Lemma \ref{lm:rb51} and Lemma \ref{lm:rb52}, respectively. After that, we give the proof of 
Lemma \ref{lm:rb5}. 
\begin{lemma}\label{lm:rb51}
\[|X(5,1)|~\leq~ 15.\]
\end{lemma}
\begin{proof}
In this case, we have $\alpha=1/5$, $K=5$. By Theorem \ref{th:s2} (1), for any $x\in X(5,1)$, 
\[|\ex|~\leq~ \frac{1-\alpha}{\ell(K,1)-\alpha}~=~\frac{1-1/5}{\ell(5,1)-1/5}~=~3.\]
By Theorem \ref{thm:ec},  there are $\binom{5}{1}=5$ equivalence classes $\ex$ in $X(5,1)$. So 
$|X(5,1)|~\leq~ 5\times 3 ~=~ 15$. 
\end{proof}

\begin{lemma}\label{lm:rb52}
\[|\{p_1, \ldots, p_5\} \bigcup X(5,2)|~\leq~ 
\begin{cases}
 275, &  23\leq r\leq 185 \\ 
 r +\lfloor (r-5)/2\rfloor,& r\geq 185.
\end{cases}
\]
\end{lemma}
\begin{proof}
Let $p_6=-\sum_{i=1}^5p_i$. 
Note for each $j=1,\ldots,5$, 
\[\langle p_j, p_6\rangle~=~\langle p_j, -\sum_{i=1}^5p_i\rangle~=~-1 - 4\times(-1/5)~=~-1/5.\]
Note also for any  $x\in X(5,2)$, by the definition of $X(5,2)$ in (\ref{eq:x}), there are $2$ or $5 - 2 =3$ positive $1/5$'s among $\langle x, p_i\rangle$ for $i= 1 \hdots 5$.
If there are $2$ positive $1/5$'s in these $5$ inner products, then
\[\langle x, p_6\rangle~=~\langle x, -\sum_{i=1}^5p_i\rangle~=~-2\times (1/5) - 3\times(-1/5)~=~ 1/5,\]
and if there are $3$ positive $1/5$'s in the $5$ inner products, then
\[\langle x, p_6\rangle~=~\langle x, -\sum_{i=1}^5p_i\rangle~=~-2\times (-1/5) - 3\times(1/5)~=~-1/5.\]
Let
\[Y = \{p_6\}\bigcup\{p_1, \ldots, p_5\} \bigcup X(5,2). \] 
Then $Y$ is equiangular with angle $1/5$ in ${\mathbb R}^r$, and we have $\km(Y)=6$ since 
 \[G\left(p_1, \ldots, p_6\right)=(1+\alpha)I-\alpha J \big\vert_{\alpha=1/5}.\]
 However, unlike in previous lemmas, $Y$ is not switching equivalent to a subset of $X$ as
  \[\km(X)=5 < 6 =\km(Y).\] 
  Thus, neither $p_6$ nor $-p_6$ is in $X$.
But it still holds by Theorem~\ref{ls1973} that, 
\[|Y|~\leq~
\begin{cases}
 276, &  23\leq r\leq 185 \\ 
 r + 1 + \lfloor (r-5)/2\rfloor,& r\geq 185.
\end{cases}
\]
So we have
\[
|\{p_1, \ldots, p_5\} \bigcup X(5,2)| ~=~ |Y|-1~\leq~ 
\begin{cases}
 275, &  23\leq r\leq 185 \\ 
 r +\lfloor (r-5)/2\rfloor, & r\geq 185.
\end{cases}
\]
\end{proof}

\noindent
{\bf Proof of Lemma \ref{lm:rb5}.}
By Formula (\ref{eq:dec5}) and Lemmas \ref{lm:rb51}--\ref{lm:rb52}, we have 
\begin{align*}
|X|~=~|X(5,1)| + |\{p_1, \ldots, p_5\} \bigcup X(5,2)|~\leq~
\begin{cases}
 290, &  23\leq r\leq 185 \\ 
 r + 15 + \lfloor (r-5)/2\rfloor, & r\geq 185.
\end{cases}~~~~~~~~~~~~~~\Box
\end{align*}

\end{appendices}

\end{document}